\documentclass[11pt]{article}
\usepackage[utf8]{inputenc}

\usepackage{amsthm}
\usepackage{amssymb}
\usepackage{amsmath}
\usepackage{amsfonts}
\usepackage{times}
\usepackage{mathtools}
\usepackage{mathrsfs}
\usepackage{color}
\usepackage{bbm}
\usepackage[colorlinks]{hyperref}
\usepackage{tikz}
\usepackage{authblk}
\usetikzlibrary{matrix}

\usepackage[paper=a4paper, top=2.5cm, bottom=2.5cm, left=2.5cm, right=2.5cm]{geometry}

\newcommand{\tO}{\tilde{O}}

\newcommand{\by}{\mathbf{y}}
\newcommand{\Mellin}{\mathfrak{M}}

\newcommand{\hmu}{\hat{\mu}}
\newcommand{\hy}{\hat{\mathbf{y}}}
\newcommand{\mo}{\omega}
\newcommand{\moy}{\mathbf{y}}
\newcommand{\thirdT}{\alpha}

\newcommand{\fracInt}{\lambda}

\newcommand{\isoperi}{\psi}

\newcommand{\area}{\mathcal{A}}
\newcommand{\transition}{\mathcal{T}}
\newcommand{\tmix}{\mathfrak{t}_{\rm{mix}}}

\newcommand{\tvdis}{{\rm{d}}_{{\rm{TV}}}}

\newcommand{\sphere}{\ensuremath{\mathbb{S}^{n}}}

\newcommand{\dirichlet}{\mathcal{E}}

\newcommand{\bb}{\mathbf{b}}

\newcommand{\defn}{:=}

\DeclarePairedDelimiterX{\infdivx}[2]{(}{)}{%
  #1\;\delimsize\|\;#2%
}

\def\qed{\hfill $\vcenter{\hrule height .3mm
\hbox {\vrule width .3mm height 2.1mm \kern 2mm \vrule width .3mm
height 2.1mm} \hrule height .3mm}$ \bigskip}

\def \RR {\mathbb R}

\def \EE {\mathbb E}

\def \Var {\mathrm{Var}}

\def \PP {\mathbb P}

\newtheorem{theorem}{Theorem}
\newtheorem{lemma}[theorem]{Lemma}

\theoremstyle{definition}
\newtheorem{definition}[theorem]{Definition}
\theoremstyle{remark}

\long\def\symbolfootnotetext[#1]#2{\begingroup
\def\thefootnote{\fnsymbol{footnote}}\footnotetext[#1]{#2}\endgroup}

\DeclareMathOperator{\Cov}{Cov}

\DeclareMathOperator{\vol}{vol}

\newcommand{\real}{\ensuremath{\mathbb{R}}}
\newcommand{\naturalnum}{\ensuremath{\mathbb{N}}}
\newcommand{\Ind}{\ensuremath{\mathbb{I}}}
\newcommand{\borel}{\ensuremath{\mathfrak{B}}}

\newcommand{\ball}{\ensuremath{\mathbb{B}^n}}
\newcommand{\Exs}{\ensuremath{{\mathbb{E}}}}
\newcommand{\Prob}{\ensuremath{{\mathbb{P}}}}

\newcommand{\Normal}{\ensuremath{\mathcal{N}}}

\newcommand{\brackets}[1]{\left[ #1 \right]}
\newcommand{\parenth}[1]{\left( #1 \right)}

\newcommand{\braces}[1]{\left\{ #1 \right \}}
\newcommand{\abss}[1]{\left| #1 \right |}
\newcommand{\angles}[1]{\left\langle #1 \right \rangle}

\newcommand{\tp}{^\top}

\newcommand{\vecnorm}[2]{\left\| #1\right\|_{#2}}

\begin{document}

\title{Hit-and-run mixing via localization schemes}

\author[1]{Yuansi Chen}
\author[2]{Ronen Eldan}
\affil[1]{Duke University}
\affil[2]{Microsoft Research, Redmond}
\date{}
\maketitle

\begin{abstract}
  We analyze the hit-and-run algorithm for sampling uniformly from an isotropic convex body $K$ in $n$ dimensions. We show that the algorithm mixes in time $\tO(n^2/ \psi_n^2)$, where $\psi_n$ is the smallest isoperimetric constant for any isotropic logconcave distribution, also known as the Kannan-Lovasz-Simonovits (KLS) constant~\cite{kannan1995isoperimetric}. Our bound improves upon previous bounds of the form $\tO(n^2 R^2/r^2)$, which depend on the ratio $R/r$ of the radii of the circumscribed and inscribed balls of $K$, gaining a factor of $n$ in the case of isotropic convex bodies. Consequently, our result gives a mixing time estimate for the hit-and-run which matches the state-of-the-art bounds for the ball walk. Our main proof technique is based on an annealing of localization schemes introduced in Chen and Eldan~\cite{chen2022localization}, which allows us to reduce the problem to the analysis of the mixing time on truncated Gaussian distributions.
\end{abstract}

\section{Introduction}
Sampling from a high dimensional distribution is a fundamental computational problem in many fields, such as Bayesian statistics, machine learning, statistical physics, and others involving stochastic models. A particularly important class of high dimensional distributions consists of uniform distributions over convex bodies. For example, the problem of sampling uniformly from a convex body is closely related to that of efficiently computing its volume, which is a fundamental problem in computer science and has been extensively studied in the last three decades (see ~\cite{dyer1991random,lovasz1990mixing,applegate1991sampling,lovasz1993random}, the survey~\cite{vempala2005geometric} and the thesis~\cite{cousins2017efficient}). Besides the volume computation, uniform sampling from a convex body can be seen as a special case of sampling from truncated Gaussian distributions, which arise naturally in Bayesian statistical models involving probit regression and censored data~\cite{albert1993bayesian,held2006bayesian}.

The \textit{hit-and-run} algorithm is a widely-used Markov-chain-based sampling method. It was introduced by Smith in 1984~\cite{smith1984efficient} and is closely related to the popular Gibbs sampler~\cite{turchin1971computation}. In the case of uniform sampling from a convex body, the hit-and-run algorithm works iteratively as follows. At each step, it starts from a point $x$ inside the convex body, chooses a uniformly distributed random direction, and then samples a point uniformly from the line segment formed by intersection of the convex body with the line segment passing through $x$ in the randomly chosen direction. This process is repeated until the chain is well-mixed.

The hit-and-run algorithm has been shown to mix rapidly for any log-concave distribution~\cite{lovasz2006fast}. In particular, the works \cite{lovasz1999hit,lovasz2006hit} show that the hit-and-run algorithm mixes in $\tO(n^2 R^2/r^2)$ from a warm start or any interior point, where $R$ and $r$ are the radii of the circumscribed and inscribed balls of $K$. In the case of \emph{isotropic} convex bodies, it was an open question whether the mixing time of the hit-and-run has to depend on the ratio $R/r$, which is typically of the order $\sqrt{n}$ (see~\cite{lee2018kannan}). By isotropic, we mean that the uniform distribution over the convex body has mean $0$ and covariance $\Ind_n$. Consequently, it was unknown whether the hit-and-run algorithm mixes in $\tO(n^2)$ steps for any isotropic convex body. The main contribution of this paper is to provide a positive answer to this question:

\begin{theorem}[main theorem, informal]
Let $\mu$ be the uniform distribution on an $n$-dimensional isotropic convex set $K$. Let $\nu$ be a measure obtained by running $t$ steps of the hit-and-run chain starting from any measure whose density with respect to $\mu$ is at most $M$. Then the total variation distance between $\mu$ and $\nu$ is at most $\epsilon$ under the condition
  \begin{align*}
    t \geq  n^2 \parenth{\frac{M \log(n) }{\epsilon}}^{O(1)}.
  \end{align*}    
\end{theorem}
A formal statement of this result can be found in Theorem~\ref{thm:hitandrun_mixing_isotropic_constant_warm} at the end of this section. Our result makes two main contributions to the existing literature on sampling from convex sets, as described next.

First, our result effectively matches the known mixing bound for the hit-and-run walk with the state-of-the-art bound for the \emph{ball walk}, another well-studied sampling algorithm. Given a convex set $K$ from which we want to produce samples, the ball walk is the Markov chain whose step is defined as follows: Given a starting point $x$, it chooses a point $y$ uniformly in a ball of fixed radius around $x$ and moves to $y$ if $y$ is in $K$; otherwise, it rejects the move and stays at $x$. It was shown in~\cite{kannan1997random} that the ball walk mixes in $\tO(n^2 R^2/r^2)$ from a warm start, where $R$ and $r$ are the radii of the circumscribed and inscribed balls of $K$. In addition to the result which depends on the ratio $R/r$, the results in~\cite{kannan1997random} also imply that the ball walk mixes in $\tO(n^2/\isoperi_n^2)$ steps for an isotropic convex body, where $\psi_n$ is the Kannan-Lovasz-Simonovits (KLS) constant~\cite{kannan1995isoperimetric}. Using the  best known bound for $\psi_n$ ~\cite{chen2021almost,klartag2022bourgain,jambulapati2022slightly} which is of order $\log^{-5} n$, the mixing time of the ball walk becomes $\tO(n^2)$. Thus, compared to the ball walk, our result gives a matching bound in terms of the dimension $n$ dependency for sampling an isotropic convex body using the hit-and-run. To the best of our knowledge, the ball walk is currently the Markov chain with the best known mixing rate for sampling a general isotropic convex body.

Second, our result is the first application of the \emph{annealing with localization schemes} technique, which was put forth in ~\cite{chen2022localization}, towards sampling from continuous distributions. The main proof strategy is to use the stochastic localization process~\cite{eldan2013thin} in order to reduce the original mixing time analysis to that of sampling from a \emph{truncated Gaussian distribution}, which is known to be well-behaved in many cases. While the high-level of the proof follows this technique, the adaptation of the technique to the hit-and-run chain seems to require substantial additional work, resulting in a framework which we hope would become relevant for other chains as well. 

Furthermore, in regard to the comparison with the ball walk, we would like to highlight that unlike the ball walk which depends on a parameter indicating the jump size (which corresponds to the radius of the ball), the hit-and-run chain is more canonical in the sense that it does not depend on a ``scale'' parameter. Given a membership oracle for an isotropic convex body, it is often not clear what the optimal choice of jump size should be, with no extra information about the convex body one needs to assume the worst case (taking a jump size of $n^{-1/2}$). The hit-and-run chain, by construction, chooses the (in a sense) optimal jump size for each point. Therefore, in light of our bound, while for the worst-case example of an isotropic convex body the two algorithms have comparable performance, there are cases of convex bodies for which the hit-and-run algorithm is strictly faster than the ball-walk (this is essentially true whenever the inscribed radius is much bigger than order $1$). In other words our bound effectively establishes that the hit-and-run walk is always at least as good as the ball walk for sampling for isotropic convex bodies, but in some cases it's actually strictly better. 

\paragraph{Remark.} It is known that the hit-and-run also mixes rapidly from a \emph{cold start}, namely it can be started from any single interior point of the convex body~\cite{lovasz2006hit}. Specifically, ~\cite{lovasz2006hit} shows that the mixing time in this case is $\tO \left (n^2 \frac{R^2}{r^2} \text{polylog}(M) \right)$, hence the dependence on the warmness parameter $M$ is poly-logarithmic rather than polynomial. This leads to the open question of whether one can obtain a $\tO(n^2 \text{polylog}(M))$ mixing bound for the hit-and-run chain in the case of isotropic convex bodies.

\subsection{Related work}
Sampling uniformly from a convex body is a well-studied problem in the literature. The first polynomial-time algorithm for this problem was proposed by Dyer, Frieze and Kannan~\cite{dyer1991random}. Many subsequent algorithms and improved mixing times have been developed~\cite{lovasz1993random,kannan1997random}. The best known mixing time for sampling a general convex body, whose circumscribed and inscribed balls have radii $R$ and $r$, respectively, from a warm start is $\tO(n^2 R^2/r^2)$, achieved by the ball walk in~\cite{kannan1997random}. The results in \cite{kannan1997random} also imply that the ball walk mixes in $\tO(n^2/\isoperi_n^2)$ for an isotropic convex body, where $\psi_n$ is the Kannan-Lovasz-Simonovits (KLS) constant~\cite{kannan1995isoperimetric}. Sampling from a convex body is closely related to the problem of efficiently computing its volume. Faster mixing often leads to faster volume computation algorithms. For related literature on volume computation, we refer the readers to~\cite{kannan1997random,lovasz2006simulated,cousins2018gaussian,jia2021reducing} and the references therein.

Since a convex polytope is a special case of a convex body, the problem of sampling uniformly from a polytope is covered by algorithms which sample from general convex bodies. However, the additional structure of a polytope also allows for new algorithms with provably better mixing times~\cite{kannan2009random,lee2017geodesic,lee2018convergence,chen2018fast,mangoubi2019faster,laddha2020strong}.

Specifically for the hit-and-run, \cite{lovasz1999hit} shows that it mixes rapidly from a constant-warm start. That is, in $\tO(n^2 R^2/r^2)$ steps, the total variation distance between the current distribution and the stationary distribution is at most a small constant. Here $R$ and $r$ are the radii of the circumscribed and inscribed balls of the convex body, respectively. This bound has the same order of magnitude as for the ball walk, so hit-and-run is no worse when we assume a bound on $R/r$.

While the ball walk is known to mix slowly if started at a corner of the convex body, the hit-and-run is known to mix rapidly from any single interior point. \cite{lovasz2006hit} shows that the hit-and-run mixes in $\tO(n^2 R^2/r^2 \text{polylog}(M))$ from any $M$-warm start. So even if $M$ is of order $n$ to a fixed degree, the mixing time remains in the same order. This line of work has been extended in~\cite{lovasz2006fast} for sampling general logconcave distributions with smoothness assumptions. 

In terms of proof techniques, our proof uses the stochastic localization process~\cite{eldan2013thin} to transform the original uniform distribution to a simpler truncated Gaussian distribution. The idea of using localization processes towards proving mixing bounds for Markov chains is summarized in the framework introduced in ~\cite{chen2022localization}. Localization schemes have also been applied to the analysis of high dimensional distributions that arise in functional analysis, convex and discrete geometry, combinatorics and mathematical physics (see the survey~\cite{eldan2022analysis}). In particular, the use of the stochastic localization process has led to the near-resolution of the Kannan, Lov\'asz and Simonovits conjecture, Bourgain's hyperplane conjecture and the thin-shell conjecture in convex geometry~\cite{chen2021almost,klartag2022bourgain,jambulapati2022slightly}.

\subsection{Formal statement of the problem and results}
Next, we make the required definitions towards the precise statement of our main result. We assume that $K \subset \RR^n$ is a convex body, which is a compact convex set with nonempty interior.
\paragraph{Target distribution.} The distribution from which we want to sample is the uniform measure on the convex body $K \subset \real^n$,
\begin{align*}
  \mu(x) \propto \mathbf{1}_{K}(x), \forall x \in \real^n.
\end{align*}
\paragraph{Isotropic position.} We say a measure $\nu$ on $\real^n$ is \textit{isotropic} if 
\begin{align*}
  \Exs_{X \sim \nu}[X]  = 0 \text{ and } \Var_{X \sim \nu}[X] = \Ind_n.
\end{align*}
A convex body is called isotropic if the corresponding uniform measure is isotropic. 

\paragraph{Hit-and-run Markov chain for a convex body.} The \textit{hit-and-run} chain is defined by the following transition step: given the current state $u \in K$, we generate unit vector $\theta \in \real^n$ which is sampled from the uniform measure on the unit sphere and consider the line $\ell := \{u + t \theta; ~ t \in \RR \}$. The next point is then chosen uniformly from the segment $\ell \cap K$.

\paragraph{Hit-and-run-transition kernel for a general target density.} 
Next, we give a more general definition of the hit-and-run chain. Given a density $\nu$ with respect to the Lebesgue measure on $\RR^n$, we denote by $P_{u \to \cdot}(\nu)$ the hit-and-run one-step transition kernel with starting point $u \in \real^n$ with respect to the underlying (target) density $\nu$, which is defined as follows: For any measurable set $A \subseteq \real^n$, we have
\begin{align}
  \label{eq:hit-and-run_transition}
  P_{u \to A}(\nu) \defn \frac{2}{n \pi_n} \cdot \int_{A} \frac{\nu(x) dx}{\nu(\ell_{ux}) \abss{u-x}^{n-1}},
\end{align}
where $\pi_n$ is the volume of the unit ball $\pi_n = \vol(\ball) = \frac{\pi^{n/2}}{\text{Gamma}(n/2+1)}$, with $\text{Gamma}$ as the gamma function and $\nu(\ell_{ux})$ is the integral of $\nu$ along the line $\ell_{ux}$ through $u$ and $x$. Specifically, for $\ell_{ux}$ the line through $u$ and $x$, we define 
\begin{align*}
  \nu(\ell_{ux}) \defn \int_{\ell_{ux}} \nu(v) \mathcal{H}_1(dv)
\end{align*}
where $\mathcal{H}_1$ is the one-dimensional Hausdorff measure. It is straightforward to check that when taking $\nu$ to be the uniform measure on $K$, this definition identifies with the previous one. Additionally, it is not hard to see that the above chain is reversible, and its stationary distribution is $\nu$~\cite{lovasz2006hit}. 
\paragraph{Lazy chain.} Given a Markov chain with transition kernel $P$. We define its lazy variant $P^{\text{after-lazy}}$, which stays in the same state with probability at least $\frac{1}{2}$, as
\begin{align*}
  P^{\text{after-lazy}}_{x \to S} = \frac{1}{2} \delta_{x \to S} + \frac{1}{2} P_{x \to S}.
\end{align*}
Here $\delta_{x \to \cdot}$ is the Dirac distribution at $x$. 
Since the lazy variant only slows down the convergence rate by a constant factor, we study lazy Markov chains in this paper for its convenience in
theoretical analysis.

Next we introduce a few notions in order to quantify the mixing time of the hit-and-run algorithm. 
\paragraph{Total-variance distance.} We denote the total variation (TV) distance between two probability distributions $\mathcal{P}_1, \mathcal{P}_2$ by 
\begin{align*}
  \tvdis(\mathcal{P}_1, \mathcal{P}_2) = \sup_{A \in \borel(\real^n)} \abss{\mathcal{P}_1(A) - \mathcal{P}_2(A)},
\end{align*}
where $\borel(\real^n)$ is the Borel sigma-algebra on $\real^n$. If $\mathcal{P}_1$ and $\mathcal{P}_2$ admit densities $p_1$ and $p_2$ respectively, we may write
\begin{align*}
  \tvdis(\mathcal{P}_1, \mathcal{P}_2) = \frac{1}{2} \int \abss{p_1(x) - p_2(x)} dx.
\end{align*}

\paragraph{Warm start.} We say an initial distribution $\mu_{\rm{init}}$ is \textit{$M$-warm} if it satisfies
\begin{align*}
  \sup_{S \in \borel(\real^n)} \frac{\mu_{\rm{init}}(S)}{\mu(S)} \leq M. 
\end{align*}
If $M$ is a constant that does not depend on $n$, then we say $\mu_{\rm{init}}$ 
is \textit{constant $M$-warm}. 

\paragraph{Mixing time.} For an error tolerance $\epsilon \in (0, 1)$, the total variance distance $\epsilon$-mixing time of the Markov chain $P$ with initial distribution $\mu_{\rm{init}}$ and target distribution $\mu$ is defined as
\begin{align*}
  \tmix(\epsilon, \mu_{\rm{init}}, \mu) \defn \inf\braces{k \in \naturalnum \mid \tvdis\parenth{\mathcal{T}^k_P (\mu_{\rm{init}}), \mu} \leq \epsilon }. 
\end{align*}
$$~$$
With the above definitions in hand, we state our main theorem. We are interested in obtaining an upper bound for the mixing time of the hit-and-run algorithm for sampling from an isotropic density $\mu \propto \mathbf{1}_K$ in terms of the $\epsilon$-mixing time. Our main theorem reads:
\begin{theorem}
  \label{thm:hitandrun_mixing_isotropic_constant_warm}
  Let $\mu$ be the uniform distribution on an $n$-dimensional isotropic convex body $K \subset \RR^n$. There exist universal constants $C, c > 0$, such that for any $M$-warm initial distribution $\mu_{\rm{init}}$ and any error tolerance $\epsilon \in (0, 1)$ such that $n \geq c \log \frac{M}{\epsilon}$, the $\epsilon$-mixing time of the lazy hit-and-run is upper bounded as follows 
  \begin{align*}
    \tmix\parenth{\epsilon, \mu_{\rm{init}}, \mu} \leq C \frac{n^2}{  \psi_n^2} \parenth{\frac{M}{\epsilon}}^{11} \log^5 \frac{M}{\epsilon}.
  \end{align*}
\end{theorem}
The proof, together with an outline of the proof strategy, are provided in Section~\ref{sec:proof_setup}. 

\section{Preliminaries}
In this section, we introduce notation, background and preliminary results needed for our proof.
\subsection{Logconcavity and concentration}
\paragraph{Logconcave density.} We say a density $\nu$ is \textit{logconcave} if it satisfies
\begin{align*}
  \nu(x)^{\tau} \nu(y)^{1-\tau} \leq \nu(\tau x + (1-\tau) y), \text{for } x, y \in \real^n, \tau \in [0, 1].  
\end{align*}
For example, $\mu \propto \mathbf{1}_K$ with $K$ being a convex set is logconcave.

\paragraph{Cheeger's isoperimetric constant.} We define the \textit{Cheeger's isoperimetric constant} of a measure $\nu$
\begin{align*}
  \isoperi_\nu \defn \inf_{A \subseteq \real^n} \braces{\frac{\int_{\partial A} \nu}{\min\braces{\nu(A), 1 -\nu(A)}}}.
\end{align*}
And 
\begin{align*}
  \isoperi_n \defn \inf_{\nu \text{ isotropic logconcave on } \real^n} \psi_\nu.
\end{align*} 
It is known that $\psi_n \geq \log^{-5}(n)$~\cite{klartag2022bourgain}. A closely related quantity if $\kappa_n > 0$ defined as follows
\begin{align*}
  \kappa_n^2 \defn \sup_{\nu \text{ isotropic logconcave}} \sup_{\theta \in \sphere} \braces{\abss{\Exs_{X \sim \nu} \angles{X, \theta} \parenth{X \otimes X}}^2 },
\end{align*}
where the first supremum is taken over all isotropic logconcave measures on $\real^n$ and $\sphere$ is the unit sphere in $\real^n$. It is known in Eldan~\cite{eldan2013thin} that there exists a universal constant $C>0$ such that $\frac{1}{\isoperi_n^2} \leq C \log n \cdot \kappa_n^2$.


\subsection{Definitions related to geometric convexity}
Let $K$ be a convex body (compact, convex, full-dimensional convex set) in $\real^n$. Denote by $\ball(x, \tau)$ with ball with center $x$ and radius $\tau$ in $\real^n$. Let $\vol$ be the $n$-dimensional Lebesgue measure. Define 
\begin{align}
  \label{eq:def_fracInt}
  \fracInt(u, t) \defn \frac{\vol(K \cap \ball(u, t))}{\vol(\ball(0, t))},
\end{align}
the fraction of a ball of radius $t$ centered around $u$ that intersects $K$. Let $K_r$ be the set of points $x \in K$ with large $\lambda(x, 2r)$, that is,
\begin{align}
  \label{eq:def_Kr}
  K_r \defn \braces{x \in K \mid \lambda\parenth{x, 2r} \geq \frac{63}{64}}.
\end{align}
As shown in ~\cite{lovasz1999hit}, the set $K_r$ is convex thanks to the Brunn-Minkowski Theorem.

\subsection{Definitions regarding Markov chains}
We are interested in the problem of sampling from a target measure $\nu$ on $\real^n$. Given a Markov chain with transition kernel $P: \real^n \times \borel(\real^n) \to \real_{\geq 0}$ where $\borel(\real^n)$ denotes the Borel $\sigma$-algebra on $\real^n$, the $k$-step transition kernel $P^k$ is defined recursively by
\begin{align*}
  P^k_{x \to dy} = \int_{z \in \real^n} P^{k-1}_{x \to dz} P_{z \to dy}. 
\end{align*}
\paragraph{Associated transition operator.} Let $\transition_P$ denote the transition operator associated to the Markov chain. It is defined as
\begin{align*}
  \transition_P(\nu) (S) \defn \int_{y \in \real^n } d\nu(y) P(y, S), \quad \forall S\in \borel(\real^n).
\end{align*}
When $\nu$ is the distribution of the current state, $\transition_P(\nu)$ is the distribution of the next state. And $\transition_P^n(\nu) \defn \transition_{P^n}(\nu)$ is the distribution of the state after $n$ steps. 

\paragraph{Dirichlet form.} Let $\mathcal{L}_2(\nu)$ be the space of square integrable functions under the measure $\nu$. The \textit{Dirichlet form} $\dirichlet_P: \mathcal{L}_2(\nu) \times \mathcal{L}_2(\nu) \to \real_{\geq 0}$ associated with the transition kernel $P$ is given by
\begin{align*}
  \dirichlet_P(f, g) \defn \frac{1}{2} \int \int \parenth{f(x) - f(y)} \parenth{g(x) - g(y)}P_{x \to dy} d\nu(x).
\end{align*}

\paragraph{Truncated conductance.} For $s \in (0, 1)$, we define the $s$-conductance $\Phi_s$ of the Markov chain $P$ with its stationary measure $\nu$ as follows
\begin{align}
  \label{eq:def_s_conductance}
  \Phi_s (P) \defn \inf_{S: s < \nu(S) < 1-s} \frac{\int_S P_{x \to S^c} d\nu(x)}{\min\braces{\nu(S), \nu(S^c)} - s}.
\end{align}
When compared to conductance (the case $s=0$), $s$-conductance allows us to ignore small parts of the distribution where the conductance is difficult to bound. Remark that using the Dirichlet form notation, we can write
\begin{align*}
  \Phi_s (P) = \inf_{S: s < \nu(S) < 1-s} \frac{\dirichlet(\mathbf{1}_{S}, \mathbf{1}_S)}{\min\braces{\nu(S), \nu(S^c)} - s}.
\end{align*}

The following lemma by Lov\'asz and Simonovits~\cite{lovasz1993random} connects the $s$-conductance with the mixing time of a Markov chain.

\begin{lemma}[Corollary 1.5 in~\cite{lovasz1993random}]
  \label{lem:lovasz_lemma}
  Consider a reversible lazy Markov chain with transition kernel $P$ and stationary distribution $\mu$. Let $\mu_{\rm{init}}$ be an $M$-warm initial distribution. Let $0 < s < \frac{1}{2}$. Then 
  \begin{align*}
    \tvdis\parenth{\transition_P^N (\mu_{\rm{init}}), \mu} \leq Ms + M \parenth{1 - \frac{\Phi_s^2}{2}}^n .
  \end{align*}
\end{lemma}

\subsection{Background on the stochastic localization process}
\label{sub:stochastic_localization}
For a probability density $\nu$ on $\RR^n$, define
\begin{align*}
  \bb(\nu) \defn \int x \nu(x) dx,
\end{align*}
its center of mass.

Given a density $\mu$ on $\real^n$, we define the stochastic localization (SL) process~(\cite{eldan2013thin}) with a positive semi-definite control matrix $C_t$, by 
\begin{align}
  \label{eq:def_mu_t}
  \mu_t(x) \defn \frac{1}{Z(t, c_t)} \exp\parenth{c_t\tp x - \frac{1}{2} \langle B_t x, x \rangle} \mu(x),
\end{align}
where $c_t$ and $B_t$ satisfy the following stochastic differential equations:
\begin{align*}
  dc_t &= C_t dW_t + C_t^2 \bb(\mu_t) dt\\
  dB_t &= C_t^2 dt.
\end{align*}
Here $W_t$ is the standard Brownian motion and $C_t$ is any process adapted to $W_t$ which takes values in the space of $n \times n$ matrices. The existence and uniqueness of the solutions of the SDE is shown via the standard existence and uniqueness results on SDEs (see e.g. Lemma 3 in~\cite{chen2021almost}). 

It is known that $\mu_t$ satisfies the following SDE for $x \in \real^n$
\begin{align}
  \label{eq:mu_t_sde}
  d\mu_t(x) = (x-\bb(\mu_t))\tp C_t dW_t \mu_t(x).
\end{align}
Define also
\begin{align}
  \label{eq:def_A_t}
  A_t \defn \int (x - \bb(\mu_t)) (x - \bb(\mu_t))\tp \mu_t(x) dx,
\end{align}
the covariance matrix of $\mu_t$.

Note that running the SL process with a starting measure $\mu \propto \mathbf{1}_K$ results in a density at time $t$ which is a random density which has an explicit form of a truncated Gaussian on a convex set.
\paragraph{Truncated Gaussian on a convex set.} For $m > 0$ and $\beta \in \real^n$, define
\begin{align*}
  \nu_{\beta, m}(x) \defn \frac{e^{-\frac{m}{2} \abss{x - \beta}^2} \mathbf{1}_K}{\int_{\real^n} e^{-\frac{m}{2} \abss{x - \beta}^2} \mathbf{1}_K dx }.
\end{align*}
It is the Gaussian with mean $\beta \in \real^n$ and variance $\frac{1}{m}$ supported on the convex set $K$. 

In light of Eq.~\eqref{eq:def_mu_t}, we have that for all $t \geq 0$, under the choice $C_s = \Ind_n$ for all $s \in [0,t]$, the measure $\mu_t$ obtained by the SL process has the form
\begin{align} \label{eq:def_mu_t_id}
  \mu_t = \nu_{c_t/t, t}.
\end{align}

The following is a special case of a classical result by Brascamp and Lieb~\cite{brascamp2002extensions}:
\begin{theorem}
One has $\mathrm{Cov}(\nu_{\beta,m}) \preceq \frac{1}{m} \Ind_n$.
\end{theorem}
If follows that, for every $t>0$, assuming the choice $C_s = \Ind_n$ for $s \in [0,t)$, one has almost surely
\begin{equation} \label{eq:COV-bl}
A_t \preceq \frac{1}{t} \Ind_n.
\end{equation}

\subsection{Other notation}
We use $\gamma: \real \to \real_+$ to denote the standard Gaussian density function, and $\Gamma: \real \to [0, 1]$ to denote the standard Gaussian cumulative density function. We use $\Normal(b, \sigma^2)$ to denote the Gaussian measure with mean $b$ and variance $\sigma^2$. For $p$ and $\tilde{p}$ two measures, $p * \tilde{p}$ denotes the convolution of the two. 

We use big-O notation $O(\cdot)$ to denote asymptotic upper bounds which ignore all constants. For example, we write $g_1(n) = O(g_2(n))$ if there exists a universal constant $c > 0$ such that $g_1(n) \leq c g_2(n)$ when $n$ is larger than a universal constant. We use $\tO(\cdot)$ to denote asymptotic upper bounds which ignore both constants and poly-logarithmic factors on the parameters involved.

\section{Proof of the main theorem}
\label{sec:proof_setup}
We prove Theorem~\ref{thm:hitandrun_mixing_isotropic_constant_warm}, by bounding the $s$-conductance, via Lemma \ref{lem:lovasz_lemma}. Roughly speaking, for constant $M$-warm initial distribution, taking 
\begin{align*}
  s = \frac{\epsilon}{2M}, n \geq \frac{2}{\Phi_s^2} \log \frac{2M}{\epsilon}
\end{align*} 
results in a mixing time of $\frac{2}{\Phi_s^2} \log \frac{2M}{\epsilon}$. As a consequence, the main focus of the proof is in lower bounding the $s$-conductance.

Unlike~\cite{lovasz1999hit}, we do not bound the $s$-conductance directly, as that requires an argument that depends in rather intricate ways on the geometry of the convex set $K$. Instead, we employ an annealing of localization schemes introduced in \cite{chen2022localization} which attempts to reduce the analysis to a simpler case in which the measure is localized, in the sense that the uniform measure on $K$ is multiplied by a Gaussian density with small variance.

This is done by considering the stochastic localization (SL) process defined in subsection \ref{sub:stochastic_localization} to the measure $\mu$. Given $\mu \sim \mathbf{1}_{K}$, we consider the process $(\mu_t)_{t\geq0}$ defined by equation \eqref{eq:def_mu_t} with the choice $C_t=\Ind_n$ up to time $T$. We fix the choice of the time $T=n$.

The annealing technique boils down to the following two main steps:
\begin{enumerate}
  \item Show that for a fixed set $E$ whose measure is bounded away from $0$ and $1$, the quantity $\mu_t(E)$ is also bounded away from those values with non-negligible probability. This behavior is referred to in \cite{chen2022localization} as \emph{approximate conservation of variance}. Under this condition, the conductance for the transition kernel at time $0$ can be lower-bounded in terms of the conductance of the transition kernel at time $T$, so that one only needs to give a lower bound on the latter quantity.
  \item Bound the conductance of the hit-and-run chain which corresponds to the measure $\mu_T$ which, according to Eq.~\eqref{eq:def_mu_t_id}, is a Gaussian with variance $\tfrac{1}{T} \Ind_n$ restricted to the set $K$. This is an easier task than the analysis of the original conductance, since this measure is typically localized well-inside the convex body $K$, so that a step of hit-and-run is hardly affected by its boundary.
\end{enumerate}
The first of the two steps highlighted above is captured by the following lemma.
\begin{lemma}
  \label{lem:approximate_conserv_variance_overall}
  Let $K$ be an isotropic convex body in $\RR^n$ and let $\mu_T$ be defined as above. Let $\zeta > 0$ and let $K_r$ be a convex subset of $K$ with $\mu(K_r) \geq 1- \zeta/100$. Then for any $E \subset K_r$ whose measure $\mu(E) = \xi \in (0, 1/2]$ satisfies $\zeta \leq c \xi^2/\sqrt{\log(10^4/\xi)}$, we have that
  \begin{align*}
    \PP \left ( \mu_T(E)(\mu_T(K_r) - \mu_T(E)) \geq \frac{1}{16} \cdot \mu(E) (\mu(K_r) - \mu(E)) \right ) \geq \frac{c \xi^{3/2}}{\log \frac{1}{\xi}\sqrt{n\log(n) \kappa^2_n}},
  \end{align*}
  for a universal constant $c>0$.
\end{lemma}
The following lemma lower bounds the volume of $K_r$ used in Lemma~\ref{lem:approximate_conserv_variance_overall}. See \cite{lovasz1999hit} for a proof.
\begin{lemma}[Lemma 2 in~\cite{lovasz1999hit}]
  \label{lem:K_r_size}
  Suppose $K$ contains a unit ball. Let $\mu \propto \mathbf{1}_K$. Then
  \begin{align*}
    \mu(K_r) \geq 1 - 2\sqrt{n} r.
  \end{align*}
\end{lemma}

In order to reduce the conductance of the original Markov chain to that of $P_{\cdot \to \cdot}(\mu_T)$, we also need the following lemma. Its proof follows from ~\cite[Proposition 48]{chen2022localization} and from the fact that the hit-and-run chain is the Markov chain associated to the subspace-localization scheme described in \cite[Section 2.1]{chen2022localization}.
\begin{lemma}
  \label{lem:dirichlet_form_super_martingale}
  Consider the process $(\mu_t)_t$ defined above. Fix a function $f \in \mathcal{L}_2(\mu)$. Let $P_t = P_{\cdot \to \cdot} (\mu_t)$. Define
  \begin{align*}
    D_t \defn \dirichlet_{P_t}(f, f).
  \end{align*}
  Then $(D_t)_{t \geq 0}$ is a super-martingale.
\end{lemma}

The next lemma, which corresponds to the second step described above, gives the conductance for the hit-and-run on the ``transformed'' density $\mu_T$.
\begin{lemma}
  \label{lem:conductance_transformed_density}
  There exists a universal constant $c>0$ such that the following holds true. Let $\nu_{\beta, n}$ be a probability measure defined as a truncated Gaussian on a convex set, given by the formula
  \begin{align*}
    \nu_{\beta, n}(x) \propto e^{-\frac{n}{2} \abss{x-\beta}^2 } \mathbf{1}_K.
  \end{align*}
  Define $\Upsilon\defn \braces{u \in K\mid \abss{u-\beta} \in \brackets{\frac{1}{\sqrt{2}} , \sqrt{2}  } }$ and $\delta \defn 1 - \nu_{\beta, n}(\Upsilon)$. 
  Suppose $S_1 \cup S_2$ is a partition of $K$ and let $0 < r \leq \frac{1}{16\sqrt{n}}$. Then we have
  \begin{align*}
    \int_{S_1} P_{u \to S_2}(\nu_{\beta, n}) d\nu_{\beta, n}(x) \geq \frac{r^2 \sqrt{n}} { c } \brackets{\nu_{\beta, n}(S_1 \cap K_r) \cdot \nu_{\beta, n}(S_2 \cap K_r)  - 8 \parenth{1 + \frac{32}{r} }\delta }.
  \end{align*}
\end{lemma}
Recall from Eq.~\eqref{eq:def_mu_t_id} that $\mu_t$ is uniquely determined by the vector $c_t$, and is given by the formula $\mu_t = \nu_{c_t/t, t}$. The next lemma shows that at time $T = n$, the set $\Upsilon$ defined in Lemma~\ref{lem:conductance_transformed_density} has large measure with high probability. 

\begin{lemma}
  \label{lem:distance_pt_to_center}
  Let $(\mu_t)_t$ be the process defined above and let $(c_t)_t$ be the corresponding process which appears in Eq.~\eqref{eq:def_mu_t}, then at time $T=n$, 
  \begin{enumerate}
    \item[(i)] the random vector $c_n/n$ has the law $\mu * \Normal(0, \frac{1}{n} \Ind_n)$.
    \item[(ii)] there exists an event $\mathfrak{E} \subseteq \real^n$ with measure at least $1 - 2e^{-\frac{n}{32}}$ under the law of $c_n/n$, such that for $v \in \mathfrak{E}$, we have
    \begin{align*}
      \Prob_{\mathbf{x} \sim \nu_{v, n}}\parenth{ \frac{\sqrt2}{2} <\abss{\mathbf{x} - v} < \sqrt{2}} \geq 1 - e^{-\frac{n}{32}}.
    \end{align*}
  \end{enumerate}
\end{lemma}

\begin{proof}[Proof of Theorem~\ref{thm:hitandrun_mixing_isotropic_constant_warm}]
  Fix the following parameters
  \begin{align}
    \label{eq:main_thm_param_choices}
    s &= \frac{\epsilon}{2M}, \notag \\
    \zeta &= s^2/\sqrt{\log(10^4/s)}/10^8,  \notag \\
    r &= \frac{\zeta}{200 \sqrt{n}}, \notag \\
    T &= n.
  \end{align}
  Let $(\mu_t)_t$ be the stochastic localization process applied to the measure $\mu$. As highlighted above, we provide a lower bound for the $s$-conductance of $\mu$ in terms of that of $\mu_T$, and then derive a lower bound for the $s$-conductance of $\mu_T$.

  According to Theorem 4.1 in~\cite{kannan1995isoperimetric}, for $K$ isotropic in $\real^n$, there exists a point $x \in K$ such that $\ball\parenth{x, 1} \subseteq K$.
  Together with Lemma~\ref{lem:K_r_size}, for the above choice of $r$, we obtain 
  \begin{align*}
    \mu(K_r) \geq 1 - \frac{\zeta}{100}.
  \end{align*}
  Let $S_1 \cup S_2 = K$ be a partition of $K$, with $s \leq \mu(S_1) \leq \frac{1}{2}$. Let $P_t = P_{\cdot \to \cdot }(\mu_t)$ be the hit-and-run transition kernel with the target density $\mu_t$.
  To lower bound the $s$-conductance, we need to lower bound
  \begin{align*}
    \dirichlet_{P_0}(\mathbf{1}_{S_1}, \mathbf{1}_{S_1}) = \int_{S_1} P_{x\to S_2}(\mu) d\mu(x).
  \end{align*}
  Define 
  \begin{align*}
    \mathfrak{E} = \braces{v \in \real^n \mid \Prob_{\mathbf{x} \sim \nu_{v, n}} \parenth{ \frac{\sqrt2}{2} <\abss{\mathbf{x} - v} < \sqrt{2}} \geq 1 - e^{-\frac{n}{32}} }.
  \end{align*}
  According to Lemma \ref{lem:distance_pt_to_center} we have
  \begin{equation} \label{eq:Ebig}
  \Prob\parenth{\frac{c_n}{n} \in \mathfrak{E}} \geq 1 - 2e^{-\tfrac{n}{32}}.
  \end{equation}
  Applying Lemma \ref{lem:conductance_transformed_density} with $r$ as above, $\beta = c_n / n$ and $\delta = e^{-n/32}$, we obtain
  \begin{align}
\int_{S_1} P_{u \to S_2}(\mu_n) d\mu_{n}(x) & \geq \frac{r^2 \sqrt{n}} { c''' } \parenth{\mu_n(S_1 \cap K_r) \cdot \mu_n(S_2 \cap K_r)  - 8 \parenth{1 + \frac{32}{r } } e^{-n/32} } \mathbf{1}_{\frac{c_n}{n} \in \mathfrak{E}}  \nonumber \\
& \geq \parenth{ \frac{c \zeta^2}{\sqrt{n}} \mu_n(S_1 \cap K_r) \cdot \mu_n(S_2 \cap K_r) - c' \sqrt{n} e^{-n/32}} \mathbf{1}_{\frac{c_n}{n} \in \mathfrak{E}}. \label{eq:condbound}
  \end{align}
  We have
  \begin{align*}
    &\quad \int_{S_1} P_{x\to S_2}(\mu) d\mu(x) \\
    &\overset{(i)}{\geq} \Exs\brackets{\int_{S_1} P_{x\to S_2}(\mu_n) d\mu_n(x)} \\
    &\overset{ \eqref{eq:condbound} }{\geq} \Exs \brackets{ \parenth{ \frac{c \zeta^2}{\sqrt{n}} \mu_n(S_1 \cap K_r) \cdot \mu_n(S_2 \cap K_r) - c' \sqrt{n} e^{-n/32}} \mathbf{1}_{\frac{c_n}{n} \in \mathfrak{E}} } \\
    &\overset{(ii)}{\geq}   \frac{c \zeta^2}{\sqrt{n}} \parenth{\frac{ c'' s^{3/2}}{\log \frac1s \sqrt{n \log (n) \kappa_n^2}}  - 2 e^{-n/32}}  \cdot \mu(S_1 \cap K_r) \cdot \parenth{\mu(K_r) - \mu(S_1 \cap K_r)} - c' \sqrt{n} e^{-n/32} \\
    &\overset{(iv)}{\geq} C \cdot \frac{s^{5.5}}{n \log^{2} (\frac1s) \kappa_n\sqrt{ \log(n)}} \brackets{\min\braces{\mu(S_1), \mu(S_1^c)} - s/2},
  \end{align*}
  where $c, c' c'', c''', C$ are all universal constants. Here (i) follows from Lemma~\ref{lem:dirichlet_form_super_martingale}, which claims that $\parenth{\dirichlet_{P_t}(\mathbf{1}_{S_1}, \mathbf{1}_{S_1})}_{t \geq 0}$ is a super-martingale.  (ii) follows from Lemma~\ref{lem:approximate_conserv_variance_overall} and Eq.~\eqref{eq:Ebig}. (iv) follows because there exists a constant $c>1$ such that for $n \geq c \log(\frac{M}{\epsilon})$, meaning that $\sqrt{n \log (n )} \kappa_n e^{-n/32} \ll s^2$. The above calculation establishes a lower bound on the $s$-conductance
  \begin{align*}
    \Phi_s \geq C \frac{s^{5.5}}{n \log^{2} (\frac1s) \kappa_n\sqrt{ \log(n)}}. 
  \end{align*}
  According to~\cite{klartag2022bourgain}, $\frac{1}{\psi_n^2} \leq C \log(n) \kappa_n^2$.
  We conclude with Lemma~\ref{lem:lovasz_lemma}.
\end{proof}
\paragraph{Overview of the rest of the proof}
In the rest of the proof, in Section~\ref{sec:approx_conserv} we prove Lemma~\ref{lem:approximate_conserv_variance_overall}, which shows the approximate conservation of variance of any indicator function. In Section~\ref{sec:conductance_for_transformed} we prove Lemma~\ref{lem:conductance_transformed_density} which shows the conductance of the hit-and-run on the transformed density, and Lemma~\ref{lem:distance_pt_to_center} which shows that with high probability most of the mass of the transformed density is concentrated in a shell.

\section{Approximate conservation of variance}
\label{sec:approx_conserv}
The goal of this section is to prove Lemma \ref{lem:approximate_conserv_variance_overall}. To this end, we fix a set $E \subset \RR^n$; our aim is to show that we non-negligible probability we have both that $\mu_T(E)$ is bounded away from $0$ and $1$ and that $\mu_T(K_r)$ is close to $1$.

We have not been able to show this directly with respect to the SL process with the choice $C_t = \Ind_n$. Instead, we consider a different choice of driving matrix $C_t$ and stopping time (described below) which on one hand makes the analysis more tractable and on the other hand has the resulting distribution of the random measure being the same as that of $\mu_T$. 

\subsection{Construction of the three-stage process} \label{def:the_3_stage_SL}
The driving matrix $C_t$ is chosen so that the process has three different stages, as follows.

  Given $\mu \sim \mathbf{1}_{K}$, $E \subset \real^n$ such that $\mu(E) = \xi \in (0, 1/2]$, we run three stages of SL to obtain $(\hmu_t)_{t\geq0}$ as follows
  \begin{itemize}
    \item \textbf{Stage 1:} Starting from $\mu$, run SL with $C_t=I_n$ from time $0$ to time $T_1 \defn \frac{\xi}{\kappa_n^2\log(n)}$ to obtain $(\hat \mu_t)_{t \in [0,T_1]}$.
    \item \textbf{Stage 2:} Starting from $\hmu_{T_1}$, we choose a driving matrix $C_t$ which satisfies the following. There is a stopping time $T_2$ so that:
    \begin{enumerate}
        \item One has $ \hat \mu_{T_1}(E) = \hat \mu_{T_2}(E)$ almost surely.
        \item The matrix $n\Ind_n - \int_{0}^{T_2} C_t^2 dt$ is positive definite and its rank is almost surely at most $1$.
    \end{enumerate}    
    \item \textbf{Stage 3:} Starting from $\hmu_{T_2}$, and defining $n\Ind_n - \int_{0}^{T_2} C_t^2 dt = \lambda_1 \theta \theta \tp$ where $|\theta| = 1$, run SL with $C_t = \theta \theta \tp$ for a time period $n - \lambda_1$. 
\end{itemize}

Note that the driving matrix in Stage 2 is defined only implicitly. The exact choice of driving matrix which satisfies the two conditions of this stage is of no consequence for the rest of the proof, rather it is only important to establish to existence of such a choice. Roughly speaking these two conditions can be obtained by choosing $C_t= \text{Proj}_{H_t \cap v_t^\perp}$, where $v_t = \int_E (x- \bb(\hmu_t)) d\hmu_t(x)$ and $H_t$ is the image of the matrix $n\Ind_n - \int_0^t C_s^2 ds$. We refer the reader to \cite[Lemma 2]{eldan2021spectral} for the exact construction.

Observe that it follows from the definition of the process that, at the end of Stage $3$, one has almost surely 
\begin{equation} \label{eq:SLequiv}
\int_0^{T_3} C_t^2 dt = n \Ind_n.    
\end{equation}
Since $\hmu_{t}$ is a martingale and $T_3$ is a stopping time, applying the optional stopping theorem yields
\begin{align} \label{eq:marT3}
  \Exs[\hmu_{T_3}(x)] = \mu(x), ~~ \forall x \in \real^n.
\end{align}
According to formula \ref{eq:def_mu_t_id}, there exists a random variable $\hy_{T_3}$ on $\real^n$ such that $\hmu_{T_3}$ takes the form
\begin{align*}
  \hmu_{T_3}(x)  = \nu_{\hy_{T_3}, n}(x) = \frac{1}{Z_{T_3} (\hy_{T_3})} \exp\parenth{ - \frac{n}{2} \abss{x - \hy_{T_3}}^2 } \mu(x), ~~~ \forall x \in \real^n.
\end{align*}
where $Z_{T_3} (\hy_{T_3})$ is the normalizing constant. 
Letting $p_{\hy_{T_3}}$ be the distribution of $\hy_{T_3}$, Eq.~\eqref{eq:marT3} implies that
\begin{align*}
  \int_{\real^n} \nu_{y, n}(x) p_{\hy_{T_3}}(dy) = \mu(x), ~~ \forall x \in \real^n.
\end{align*}
On the other hand, applying the same argument for $\mu_T$ which was obtained via the SL process with driving matrix $C_t = \Ind_n$, there exists a random variable $\by_n$ such that, almost surely
$$
\mu_T(x) = \frac{1}{Z(\by_{n})} \exp\parenth{ - \frac{n}{2} \abss{x - \by_{n}}^2 } \mu(x) = \nu_{\by_{n}, n}(x), ~~ \forall x \in \real^n.
$$
Denoting the law of $\by_n$ by $p_{\by}$, the martingale property yields
\begin{align*}
  \int_{\real^n} \nu_{y, n}(x) p_{\by}(dy) = \mu(x), ~~ \forall x \in \real^n.
\end{align*}
We conclude that
\begin{equation} \label{eq:processes_are_same}
\int_{\real^n} \nu_{y, n}(x) p_{\hy_{T_3}}(dy) = \int_{\real^n} \nu_{y, n}(x) p_{\by}(dy) = \mu(x), ~~ \forall x \in \real^n.    
\end{equation}
The following lemma shows that $p_{\hy_{T_3}} = p_{\by}$, which implies that $\mu_T$ and $\hat \mu_{T_3}$ have the same distribution. Its proof is provided in Subsection~\ref{sub:identify_3_stage_SL}.
\begin{lemma}
  \label{lem:identification_of_marginal_on_Y}
  Given $\mu(x) \sim \mathbf{1}_K$ the uniform distribution on a bounded convex set $K \subset \real^n$. If there exists a density $p$ on $\real^n$, such that 
  \begin{align*}
    \mu(x) = \int \nu_{y, n}(x) p(y) dy, \forall x \in \real^n,
  \end{align*}
  then $p$ is uniquely defined almost everywhere. 
\end{lemma}

Since the statement of Lemma \ref{lem:approximate_conserv_variance_overall} only depends on the random measure $\mu_T$ and is oblivious of the path leading to it, we can prove this lemma via the analysis of the process $(\hat \mu_t)_{t \in [0,T_3]}$. The rest of the proof therefore boils down to the next lemma, proven in Subsection~\ref{sub:approx_conserv_var_for_3stage_SL}.
\begin{lemma}
  \label{lem:approx_conserv_var_at_time_n}
  Suppose that $\mu = \mathbf{1}_K$ is isotropic and logconcave and that $\mu(K_r) \geq 1-\zeta/100$. Let $E \subset K_r$ satisfy $\mu(E) = \xi \in (0, 1/2]$ with $\zeta \leq \xi^2/\sqrt{\log(10^4/\xi)}/10^8$. Let $(\hat \mu_t)_{t}$ be the 3-stage process defined above. Then,
  \begin{align*}
    \Prob \parenth{\hat \mu_{T_3}(E)(\hat \mu_{T_3}(K_r) - \hat \mu_{T_3}(E)) \geq \frac{1}{16} \cdot \mu(E) (\mu(K_r) - \mu(E))} \geq \frac{c \xi^{3/2}}{\log (\frac1\xi) \sqrt{n \kappa_n^2 \log(n)}},
  \end{align*}
  for a universal constant $c>0$.
\end{lemma}

\begin{proof}[Proof of Lemma~\ref{lem:approximate_conserv_variance_overall}]
  Lemma~\ref{lem:approximate_conserv_variance_overall} directly follows from \eqref{eq:processes_are_same} combined with the two lemmas above.
\end{proof}

\subsection{Approximate conservation of variance for the process \texorpdfstring{$\hat \mu_t$}{mu}}
\label{sub:approx_conserv_var_for_3stage_SL}
To prove Lemma~\ref{lem:approx_conserv_var_at_time_n}, we proceed with three lemmas which deal with each stage of the stochastic process one by one. Consider the following events,
\begin{align*}
  W_1 &:= \braces{ \hmu_{T_1}(E) \in \brackets{\xi - \frac{\xi}{4}, \xi + \frac{\xi}{4}} } \cap \braces{\hmu_{T_1}(K_r) \geq 1-\frac{\zeta}{20} }, \\
  W_2 &:= \braces{ \hmu_{T_2}(E) \in \brackets{\xi - \frac{\xi}{4}, \xi + \frac{\xi}{4}} } \cap \braces{\hmu_{T_2}(K_r) \geq 1-\frac{\zeta}{10} }, \text{ and } \\
  W_3 &:= \braces{ \hmu_{T_3}(E) \in \brackets{\frac{\xi}{2} , \frac{3}{4}} } \cap \braces{\hmu_{T_3}(K_r) \geq \frac{7}{8} }.
\end{align*}
\begin{lemma}
  \label{lem:stage1}
  For $\mu(E) = \xi \in (0, 1/2]$ and $\mu(K_r) \geq 1 - \zeta/100$, there exists a universal constant $C > 0$ such that at the end of Stage 1 (meaning, for $T_1 = \frac{\xi}{C \kappa_n^2 \log n}$), we have
  \begin{align*}
    \Prob\parenth{W_1 } \geq 0.6.
  \end{align*}
\end{lemma}
\begin{lemma}
  \label{lem:stage2}
  We have
  \begin{align*}
    \Prob\parenth{W_2 |~ W_1} \geq 0.5.
  \end{align*}
\end{lemma}

\begin{lemma}
  \label{lem:stage3}
  Under the assumption $\zeta \leq \xi^2/\sqrt{\log(10^4/\xi)}/10^8$, we have
  \begin{align*}
    \Prob\parenth{ W_3 | ~ W_2 } \geq \frac{c}{\log(\frac{1}{\xi})}\sqrt{\frac{T_1}{n}},
  \end{align*}
  for a universal constant $c>0$.
\end{lemma}

\subsubsection{Analysis of Stage 1}
The proof of Lemma \ref{lem:stage1} relies on the analysis developed in recent years around the KLS conjecture (see e.g.~\cite{chen2021almost,klartag2022bourgain}). We apply the following upper bound on the operator norm of the covariance matrix $A_t$, proven by Klartag and Lehec:
\begin{lemma}[Lemma 5.2 in \cite{klartag2022bourgain}]
  \label{lem:A_t_op_Klartag_Lehec}
  For every $T \leq (C \kappa_n^2 \log n)^{-1}$ we have
  \begin{align*}
    \Prob(\vecnorm{A_t}{2} \geq 2 \text{ for } 0 \leq t \leq T) \leq \exp\parenth{-\frac{1}{CT}},
  \end{align*}
  where $C$ is a universal constant.
\end{lemma}
While Lemma 5.2 in \cite{klartag2022bourgain} only shows the result for a fixed $t \leq T$, it is not hard to see that, using Doob's inequality, the same proof can be generalized to the case of all $t \in [0,T]$ with a small modification on the constant $C$.

Equipped with Lemma~\ref{lem:A_t_op_Klartag_Lehec}, Lemma~\ref{lem:stage1} then follows from a simple stochastic calculus. 
\begin{proof}[Proof of Lemma~\ref{lem:stage1}]
  Let $g_t = \hmu_t(E)$. We have
  \begin{align*}
    d g_t = \int_E (x - \bb(\hmu_t))\tp dW_t \hmu_t(x) dx.
  \end{align*}
  Its quadratic variation is
  \begin{align*}
    d [g]_t &= \abss{\int_E (x- \bb(\hmu_t)) \hmu_t(x) dx}^2 dt \\
    &\leq \vecnorm{A_t}{2} dt,
  \end{align*}
  where $A_t$ is the covariance  matrix of $\bb(\hmu_t)$. We have 
  \begin{align*}
    \Prob\parenth{\hmu_{t}(E) \in \brackets{\xi - \frac{\xi}{4}, \xi + \frac{\xi}{4}}} &= \Prob\parenth{\tilde{W}_{[g]_t} \in \brackets{- \frac{\xi}{4}, \frac{\xi}{4} } } \\
    &\geq 0.9 - \Prob\parenth{\int_0^t \vecnorm{A_\tau}{2} d\tau > \frac{\xi}{64}}
  \end{align*}
  Applying Lemma~\ref{lem:A_t_op_Klartag_Lehec} on the operator norm of $A_t$ with $t \leq \frac{\xi}{C \kappa_n^2 \log n}$, it follows that
  $$
  \Prob\parenth{\hmu_{T_1}(E) \in \brackets{\xi - \frac{\xi}{4}, \xi + \frac{\xi}{4}}} \geq 0.8.
  $$
  Next, since $\hat \mu_t(K_r)$ is a martingale, we have $\EE[\hat \mu_{T_1}(K_r)] \geq 1-\zeta/100$. Since $\hat \mu_{T_1}(K_r) \leq 1$ almost surely, we can use Markov's inequality to conclude that
  $$
  \Prob \parenth{ \hat \mu_{T_1}(K_r) \leq 1-\zeta/20} \leq 0.2.
  $$
  This concludes the lemma.
\end{proof}

\subsubsection{Analysis of Stage 2}
In the second stage of the process, the measure of $E$ is kept constant almost surely, so Lemma \ref{lem:stage2} boils down to a simple application of Markov's inequality.
\begin{proof}[Proof of Lemma~\ref{lem:stage2}]
  In the second stage, by construction, we have 
  \begin{align*}
    \hmu_{T_2}(E) = \hmu_{T_1}(E).
  \end{align*}
  Additionally, since $\hmu_t(K_r)$ is a martingale, for $t\in [T_1, T_2]$, we have 
  \begin{align*}
    \Exs [\hmu_t(K_r) \mid W_1] = \mu_{T_1}(K_r) \geq 1 - \zeta/20.
  \end{align*}
  Then $\Prob(\hmu_t(K_r) \geq 1-\zeta/10) + (1-\zeta/10) \parenth{ 1- \Prob(\hmu_t(K_r) \geq 1-\zeta/10)} \geq 1 - \zeta/20$, which results in
  \begin{align*}
    \Prob(\hmu_t(K_r) \geq 1-\zeta/10) \geq 0.5.
  \end{align*}
  We conclude by applying the union bound. 
\end{proof}

\subsubsection{Analysis of Stage 3}
In this section we prove Lemma \ref{lem:stage3}. Here is the main observation: since the driving matrix $C_t$ is fixed to be the rank-$1$ matrix $\theta \theta^T$ between time $T_2$ and $T_3$, the SL process during that time only depends on the marginal of the measure $\hat \mu_{T_2}$ onto the direction $\theta$, which means that the proof boils down to the analysis of a one-dimensional SL process. This analysis, however, is quite long and technical and requires a few lemmas on properties of one-dimensional logconcave measures summarized in Appendix~\ref{app:properties_of_logconcave}.

Recall that we condition on the event $W_2$ which amounts to
\begin{align*}
  \hmu_{T_2}(E) \in \brackets{\xi - \frac{\xi}{4}, \xi + \frac{\xi}{4}} \text{ and } \hmu_{T_2}(K_r) \geq 1 - \zeta/10,
\end{align*}
For a unit vector $\theta$ (obtained by Stage 2) the third stage of the process runs the SL process with a control matrix $C_t = \theta \theta\tp$  time period $\thirdT := n-\lambda_1$. That is, for $t\geq T_2$, 
\begin{align*}
  d \hmu_t(x) = (x- \bb(\hmu_t))\tp \theta \theta\tp dW_t \hmu_t(x).
\end{align*}
Define 
\begin{align*}
  \sigma^2 \defn \theta\tp A_{T_2} \theta = \theta\tp \mathrm{Cov}(\hmu_{T_2}) \theta,
\end{align*}
the variance of the starting measure in the direction of $\theta$. As a result of the Brascamp and Lieb inequality~\cite{brascamp2002extensions}, we have $A_{T_2} \preceq B_{T_2}^{-1}$, and hence 
\begin{align}
  \label{eq:third_SL_stage_sigma^2_bound}
  \sigma^2 \leq \frac{1}{T_1}.
\end{align}

For $t \in [0, \thirdT]$, define $\mo_t$ to be the density on $\real$ obtained by taking the push-forward of $\hmu_{T_2+t}$ via $x \mapsto \frac{1}{\sigma} \cdot x \tp \theta$. That is, 
\begin{align}
  \label{eq:def_mo_1d_pushforward}
  \mo_t(z) \defn \int_{H(z)} \hmu_{T_2 + t}(x) dx,
\end{align}
where, for $z \in \RR$, $H(z) := \{x \in \RR^n;~ \theta^T x = z \sigma \}$ is defined as the fiber corresponding to the value $z$. For a subset $S \subseteq \real^n$, define 
\begin{align}
  \label{eq:def_h_S}
  h_S(z) \defn \begin{cases} \displaystyle
    \frac{\int_{H(z)}  \mathbf{1}_{ \{x \in S \} } \hat \mu_{T_2}(x) dx}{\int_{H(z)} \hmu_{T_2}(x) dx} & \text{ if } \mo_0(z) \neq 0, \\
    0 & \text{otherwise.}
  \end{cases} 
\end{align}
Observe that $h_S(z) \in [0, 1]$ for all $z$. Note that the value of $h_S$ remains unchanged if the term $\hmu_{T_2}$ is replaced with $\hmu_{T_2+t}$ in the above formula, for any $t \in [0, \thirdT]$. This is because
\begin{align*}
  \hmu_{T_2+t}(x) \propto \hmu_{T_2}(x) \exp\parenth{ - t (\theta \tp x)^2 + \tilde{c}_t \theta \tp x},
\end{align*}
for some $\tilde{c}_t$. This means that for a fixed value of $z$ the points in $H(z)$ are all multiplied by the same factor. The definition of $h_S$ allows us to identify $\mo_t(h_S)$ with $\hmu_{T_2 +t}(S)$ as shown in the following lemma.
\begin{lemma}
  \label{lem:identification_nloc_vs_1loc}
  For the stochastic process $\mo_t$ defined above, for $t > 0$, we have
  \begin{align*}
    \bb(\mo_t) &= \theta \tp \bb(\hmu_{T_2+t}) / \sigma \\
    \Var(\mo_t) &= \theta \tp \Cov(\hmu_{T_2+t}) \theta / \sigma^2 \\
    \mo_t(h_S) &= \hmu_{T_2 + t}(S).
  \end{align*}
  In particular, $\Var(\mo_0) = 1$. Additionally, $(\mo_t)_{t\geq 0}$ satisfies the following stochastic differential equation
  \begin{align}
    \label{eq:mo_sde}
    d \mo_t = \sigma (z - \bb(\mo_t)) d \bar{W}_t \mo_t(z),
  \end{align}
  where $\bar{W}_t = \theta \tp W_t$ is a one-dimensional Brownian motion.
\end{lemma}
The proof of this lemma is straightforward, and is provided in Subsection~\ref{sub:additional_proofs_in_3_stage} below. Based on the above stochastic differential equation, for $t> 0$, define
\begin{align*}
  \moy_t \defn \frac{1}{t} \int_0^t \frac{1}{\sigma} d \bar{W}_s + \bb(\mo_s) ds.
\end{align*}
According to Lemma 2.1 in~\cite{eldan2013thin}, $\mo_t$ is uniquely determined by the value of $\moy_t$,  for all $t \in [0, \thirdT]$. Additionally, it takes the form
\begin{align*}
  \mo_t(z) = \frac{\mo_0(z) \exp (-t\sigma^2(z - \moy_t)^2)}{\int \mo_0(\varsigma) \exp (-t\sigma^2(\varsigma - \moy_t)^2) d\varsigma }.
\end{align*}
An application of Theorem 2 in~\cite{el2022information} shows that $\moy_{\thirdT}$ has the law $\rho$, where 
\begin{align}
  \label{eq:def_rho}
  \rho \defn \mo * \Normal\parenth{0, \frac{1}{\thirdT \sigma^2}}.
\end{align}
With a slight abuse of notation, we introduce a deterministic density $\mo_{y, t}$ with two subindices as
\begin{align}
  \label{eq:def_mo_y_t}
  \mo_{y, t}(z) \defn \frac{\mo_0(z) \exp\parenth{-t\sigma^2(z - y)^2}}{\int \mo_0(\varsigma) \exp\parenth{-t\sigma^2(\varsigma - y)^2} d\varsigma }, \quad \forall z \in \real.
\end{align}
From here, we can identify $\mo_t$ with $\mo_{\moy_t, t}$. The next lemma shows that the function $y \mapsto \mo_{y, \thirdT}(h_S)$  is Lipschitz for a fixed $S \subset \real$. 

\begin{lemma}
  \label{lem:lipschitzness_of_G}
  Let $h:\RR^n \to [0,1]$. Then for all $y, \tilde{y} \in \real$, we have
  \begin{align*}
    \abss{\mo_{y, \thirdT}(h) - \mo_{\tilde y, \thirdT}(h)} \leq \sqrt{\thirdT \sigma^2} \abss{ y - \tilde{y}}.
  \end{align*}
\end{lemma}
Its proof is provided in Subsection~\ref{sub:additional_proofs_in_3_stage}. With the above notation and lemmas, we are ready to prove Lemma~\ref{lem:stage3}. It is done by discussing the two cases based on the value of $\thirdT \sigma^2$:
\begin{enumerate}
  \item $\thirdT\sigma^2 < \frac{1}{512}$
  \item $\thirdT \sigma^2 \geq \frac{1}{512}$.
\end{enumerate}
Define the events
\begin{align*}
  H_1 &:= \left \{\mo_{Y, \alpha}(h_{E}) \in \brackets{\frac12 \xi, \frac34} \right \} \cap \left \{ \mo_{Y, \alpha}(h_{K_r}) \geq \frac78 \right \}, \text{ and } \\
  H_0 &:= \left \{ \mo_{0}(h_{E}) \in \brackets{\frac{3}{4}\xi, \frac{5}{4}\xi } \right \} \cap \left \{  \mo_0(h_{K_r}) \geq 1 - \frac{\zeta}{10}  \right \}.
\end{align*}
The results in the two cases are summarized in the two lemmas below. 
\begin{lemma}
  \label{lem:stage3_case1_nu}
  If $\thirdT \sigma^2 \leq \frac{1}{512}$ and $\zeta \leq 0.1 \xi$, then 
  \begin{align*}
    \Prob_{Y \sim \rho}\parenth{ H_1 \mid H_0} \geq 0.1 \xi.
  \end{align*}
\end{lemma}
\begin{lemma}
  \label{lem:stage3_case2_nu}
  If $\thirdT \sigma^2 \geq \frac{1}{512}$, $\zeta \leq \xi^2/\sqrt{\log(10^4/\xi)}/10^8$, then there exists a universal constant $C>0$ such that
  \begin{align*}
    \Prob_{Y \sim \rho}\parenth{H_1 | H_0 } \geq \frac{C}{\log(\frac{1}{\xi})}\sqrt{\frac{T_1}{n}} \xi.
  \end{align*}
\end{lemma}
\begin{proof}[Proof of Lemma~\ref{lem:stage3}]
  Using the identification in Lemma~\ref{lem:identification_nloc_vs_1loc} between the terms in $\mo$ and those in $\hmu$, it is clear that Lemma~\ref{lem:stage3} follows from the two lemmas above.
\end{proof}

\subsubsection{Proof of Lemma~\ref{lem:stage3_case1_nu} and Lemma~\ref{lem:stage3_case2_nu}}
First, we prove Lemma~\ref{lem:stage3_case1_nu} where $\thirdT \sigma^2 \leq \frac{1}{512}$. The bound is proven by direct analysis of the process $\mo_t(h_E)$ via the stochastic differential equation~\eqref{eq:mo_sde}. 

\begin{proof}[Proof of Lemma~\ref{lem:stage3_case1_nu}]
  Let $g_t \defn \mo_t(h_E)$. According to Eq.~\eqref{eq:mo_sde}, $g_t$ satisfies
  \begin{align*}
    d g_t = \int_{\real} \sigma (z - \bb(\mo_t)) \cdot d \bar{W}_t h_E(z) \mo_t(z) dz.
  \end{align*}
  It is a martingale for $t\geq 0$, with $g_0 = w_0(h_E) = \hmu_{T_2}(E) \in [\frac{3}{4} \xi, \frac{5}{4} \xi] $. We first claim that for any $t_1 \in [0, \thirdT]$ to $\thirdT$, we have almost surely that
  \begin{align}
    \label{eq:1d_evolve_mo_g}
    g_{t_1} \in [1/2, 5/8] ~~ \Rightarrow ~~ \Prob(g_\thirdT \in [3/8, 6/8] \mid \omega_{t_1} ) \geq 0.4.
  \end{align}
  The proof of the claim~\eqref{eq:1d_evolve_mo_g} is deferred to the end. Assuming the claim for now, we complete the proof of Lemma~\ref{lem:stage3_case1_nu}.
  
  If $g_0 \in [1/2, 5/8]$, then we directly apply the claim~\eqref{eq:1d_evolve_mo_g} from time $0$ to $\thirdT$ to obtain
  \begin{align*}
    \Prob(g_\thirdT \in [3/8, 6/8]) \geq 0.4. 
  \end{align*}
  Otherwise, $g_0 \in [\frac34 \xi , 1/2)$, we define the stopping time
  \begin{align}
    \label{eq:def_stopping_tau}
    \boldsymbol{\tau} \defn \min \braces{ t \in [0, \thirdT] \mid  g_t = \frac12 \xi \text{ or } g_t = \frac{1}{2} } \wedge \alpha.
  \end{align}
  Since $(g_t)_{t\geq 0}$ is a martingale with $\Exs[g_t] = g_0 \in [\tfrac34 \xi , 1/2)$, applying the optimal stopping time theorem for bounded martingale, we have
  \begin{align*}
    \Exs[g_{\boldsymbol{\tau}}] = g_0 \in [\frac34 \xi , 1/2).
  \end{align*}
  Separating the above expectation into three cases $g_{\boldsymbol{\tau}} = \frac12 \xi$, $g_{\boldsymbol{\tau}} = \frac12$ or $\boldsymbol{\tau} = \thirdT$, we obtain
  \begin{align*}
    \Prob\left (g_{\boldsymbol{\tau}} = \frac{1}{2} \text{ or } \boldsymbol{\tau} = \thirdT \right) \geq \frac12 \xi.
  \end{align*}
  Under the above event, if $\boldsymbol{\tau} = \thirdT$, we have $g_\alpha \in (\frac12 \xi , \frac12)$. Otherwise, we apply the claim~\eqref{eq:1d_evolve_mo_g} from time $\boldsymbol{\tau}$ to $\thirdT$ to obtain
  \begin{align*}
    \Prob(g_\thirdT \in [3/8, 6/8]) &=  \Prob(g_\thirdT \in [3/8, 6/8] \mid g_{\boldsymbol{\tau}} = \frac12 \text{ or } \boldsymbol{\tau} = \thirdT )  \Prob(g_{\boldsymbol{\tau}} = \frac12 \text{ or } \boldsymbol{\tau} = \thirdT )  \\
    &\geq 0.4 \Prob(g_{\boldsymbol{\tau}} = \frac12 \text{ or } \boldsymbol{\tau} = \thirdT ) \\
    & \geq 0.2 \xi.
  \end{align*}
  
  Combining all the cases above, we conclude that $g_\thirdT \in [\frac12 \xi, \frac34]$ given $g_0 \in [\frac34 \xi, \frac54 \xi]$ with probability at least $0.2 \xi$. On the other hand, since $(\mo_t(h_{K_r}))_{t\geq 0}$ is a martingale, we have $\EE[\mo_\alpha(h_{K_r}) | H_0] \geq 1-\zeta/10$. Thus, by Markov's inequality and by the fact that $\mo_\alpha(h_{K_r}) \leq 1$ almost surely,
  \begin{align*}
    \Prob(\mo_\alpha(h_{K_r}) \geq 7/8) \geq 1-\zeta.
  \end{align*}
  Recalling that $\zeta \leq 0.1 \xi$, Lemma~\ref{lem:stage3_case1_nu} follows from a union bound. 

  \paragraph{Proof of the claim~\eqref{eq:1d_evolve_mo_g}:} The quadratic variation of $g_t$ is 
  \begin{align*}
    d[g]_t &= \abss{\int \sigma (z - \bb(\mo_t)) h_E(z) \mo_t(z) dz }^2 dt \\
    &\overset{(i)}{\leq} \int \sigma^2 (z - \bb(\mo_t))^2 \mo_t(z) dz \cdot \int h_E(z)^2 \mo_t(z) dz  dt \\
    &\leq  \sigma^2 \Var(\mo_t) dt.
  \end{align*}

(i) follows from the Cauchy-Schwarz inequality. To control $\varphi_t \defn \Var(\mo_t)$, we observe that it satisfies 
\begin{align*}
  d \varphi_t =  \int_{\real} \sigma (z - \bb(\mo_t))^3 d\bar{W}_t \mo_t(z)dz - \sigma^2 \varphi_t^2 dt. 
\end{align*}
We have $\Exs d\varphi_t \leq 0$, hence $\Exs [\varphi_t] \leq \varphi_0 \leq 1$. 

Consequently, $\Exs [g]_\thirdT \leq \frac{1}{512}$. 
Conditioned on $g_0 \in [\frac12, \frac58]$, we have 
\begin{align*}
  \Prob \left (\frac{3}{8} \leq g_\thirdT \leq \frac{6}{8} \right) &\geq \Prob \parenth{-\frac{1}{8} \leq \tilde{W}_{[g]_\alpha} \leq \frac{1}{8}} \\
  &=1 - \Prob \left (\max_{0 \leq \ell \leq \frac{1}{256} } \abss{\tilde{W}_\ell } > \frac{1}{8} \right) - \Prob \left ([g]_\thirdT > \frac{1}{256} \right) \\
  &\overset{(i)}{\geq} 1 - 4 \Prob(\tilde{W}_{\frac{1}{256}}  > \frac{1}{8}) - \Prob([g]_\thirdT > \frac{1}{256}) \\
  &\overset{(ii)}{\geq} 0.9 - 0.5 = 0.4 
\end{align*}
(i) follows from the reflection principle. (ii) follows from the tail bound of the normal distribution $\Prob_{X \sim \Normal(0, 1)}( X > 2) \leq 0.023$ and the Markov inequality. 
\end{proof}

Next, we prove Lemma~\ref{lem:stage3_case2_nu} which concerns the case that $\frac{1}{512} \leq \thirdT \sigma^2$. The main idea in the proof is to use the Lipschitz bound established in Lemma \ref{lem:lipschitzness_of_G}: Since the function $y \to \mo_{y, \thirdT}(h_{E})$ is Lipschitz and since it must attain values both close to $0$ and to $1$, we conclude that there exists an interval with non-negligible mass, on which $\mo_{y, \thirdT}(E)$ both is bounded away from both $0$ and $1$. The main caveat is then to show that $\mo_{y, \thirdT}(h_{K_r})$ can be made large at the same time. This crucially relies on the convexity of $K_r$, used in the following lemma, whose proof is found in the next subsection.
\begin{lemma}
  \label{lem:case2_nu_t_y_Kr_large}
  Suppose that (i) $\frac{1}{512} \leq \thirdT \sigma^2$, (ii) $\zeta \leq \xi^2/\sqrt{\log(10^4/\xi)}/10^8$ and (iii) $\mo_0(h_{K_r}) \geq 1- \zeta/10$. Then, there exists an interval $I \subset \real$ such that $\rho(I) \geq 1 - \xi/10$ and of length at most $100 (4 + \log\frac{1}{\xi})$, such that 
  \begin{align*}
    \mo_{y, \thirdT}(h_{K_r}) \geq 7/8, ~~ \forall y \in I.
  \end{align*}
\end{lemma}

\begin{proof}[Proof of Lemma~\ref{lem:stage3_case2_nu}]
  Let $I$ be the interval obtained by applying Lemma~\ref{lem:case2_nu_t_y_Kr_large}. We partition $I$ into three subsets
  \begin{align*}
    I_1 &\defn \braces{y \in \real \mid \mo_{y, \thirdT}(h_E) \leq 1/2 \xi} \cap I, \\
    I_2 &\defn \braces{y \in \real \mid \mo_{y, \thirdT}(h_E) \in [1/2 \xi, 3/4]} \cap I,\\
    I_3 &\defn \braces{y \in \real \mid \mo_{y, \thirdT}(h_E) \geq 3/4} \cap I.
  \end{align*}
  Since $\Exs[\mo_{Y, \thirdT}(h_E) | H_0] = \mo_0(h_E) \in [3/4\xi, 5/4\xi]$, Markov's inequality implies that
  \begin{align*}
    5/4 \xi \geq 0 + 3/4 \cdot \Prob(\mo_{Y, \thirdT}(h_E) > 3/4 | H_0).
  \end{align*}
  Recalling that the measure $\rho$ defined in Eq.~\eqref{eq:def_rho} satisfies $\rho(I) \geq 1- \frac{1}{10}\xi$, we obtain by a union bound that
  \begin{align*}
    \Prob(Y \in I_1 \cup I_2 | H_0) =  \Prob(\mo_{Y, \thirdT}(h_E) \leq 3/4 | H_0) - \frac{1}{10} \xi \geq 1 - \frac{5}{3} \xi - \frac{1}{10} \xi \geq 1/10,
  \end{align*}
  where $Y$ is $\rho$-distributed. If $\Prob(Y \in I_2 | H_0) \geq 1/20$, then we are done. Otherwise, we obtain
  \begin{align}
    \label{eq:I1_large}
    \Prob(Y \in I_1) \geq 1/20.
  \end{align}
  Similarly, since
  \begin{align*}
    3/4 \xi \leq \Exs [\mo_{Y, \thirdT}(h_E)] \leq 1 \cdot \Prob(\mo_{Y, \thirdT}(h_E) \geq 1/2 \xi) + 1/2 \xi \cdot \Prob(\mo_{Y, \thirdT}(h_E) \leq 1/2 \xi),
  \end{align*}
  we obtain
  \begin{align*}
    \Prob(Y \in I_2 \cup I_3) = \Prob(\mo_{Y, \thirdT}(h_E) \geq 1/2 \xi) - \frac{1}{10}\xi \geq \frac{1}{10} \xi. 
  \end{align*}
  If $\Prob(Y \in I_2) \geq \frac{1}{20} \xi$, then we are done. Otherwise, we obtain
  \begin{align}
    \label{eq:I3_large}
    \Prob(Y \in I_3) \geq \frac{1}{20} \xi.
  \end{align}
  Observe that the distance between $I_1$ and $I_3$ is large, because for any $y \in I_1$ and $\tilde{y} \in I_3$, we have
  \begin{align*}
    \mo_{\tilde{y}, \thirdT}(h_E) - \mo_{y, \thirdT}(h_E) \geq \frac34 - \frac12 \xi \geq \frac12.
  \end{align*}
  And according to Lemma~\ref{lem:lipschitzness_of_G}, $y \mapsto \mo_{y, \thirdT}(h_E)$ is $\sqrt{\thirdT \sigma^2} \leq \sqrt{n/T_1}$-Lipschitz as $\thirdT \leq n$ and $\sigma^2 \leq \frac{1}{T_1}$ in Eq.~\eqref{eq:third_SL_stage_sigma^2_bound}. Hence,
  \begin{align}
    \label{eq:distance_I1_I3}
    \abss{\tilde{y} - y} \geq \frac{1}{2} \sqrt{\frac{T_1}{n}}.
  \end{align}
  Let $(\mathbf{1}_I \cdot \mo)$ be the measure $\mo$ restricted on the set $I$. Finally, we consider $(\mathbf{1}_I \cdot \mo) * \Normal\parenth{0, \frac{1}{\alpha \sigma^2}}$, which satisfies a diameter isoperimetric inequality (see Theorem 4.2 in~\cite{vempala2005geometric}). Hence,
  \begin{align*}
    \rho(I_2) &\geq \frac{2 d(I_1, I_3)}{100 \parenth{1 +\log(\frac{1}{\xi}) } } \min\braces{\rho(I_1), \rho(I_3)} \\
    &\geq \frac{C}{\log(\frac{1}{\xi})}\sqrt{\frac{T_1}{n}} \xi,
  \end{align*} 
  where the inequality follows from Eq.~\eqref{eq:I1_large}~\eqref{eq:I3_large} and~\eqref{eq:distance_I1_I3}.
\end{proof}

\subsubsection{Additional proofs in the third SL stage}
\label{sub:additional_proofs_in_3_stage}
In this subsection, we complete the missing proofs of the lemmas in the previous subsection.
\begin{proof}[Proof of Lemma~\ref{lem:identification_nloc_vs_1loc}]
  Based on the $\mo_t$ definition~\eqref{eq:def_mo_1d_pushforward}, we can express the mean and variance of $\mo_t$ using $\hmu_{T_2+t}$ as follows.
  \begin{align*}
    \bb(\mo_t) &= \int z \mo_t(z) dz \\
    &= \int \int_{\theta \tp x / \sigma  = z} \theta\tp x / \sigma  \hmu_{T_2+t}(x) dx dz \\
    &=  \theta \tp \bb(\hmu_{T_2+t})/\sigma \\
    \Var(\mo_t) &= \int \parenth{z - \bb(\mo_t)}\parenth{z - \bb(\mo_t)}  \mo_t(z) dz  \\
    &= \int \int_{\sigma \theta \tp x = z} \frac{1}{\sigma^2} \theta \tp (x - \bb(\hmu_t))(x - \bb(\hmu_t))\tp \theta \hmu_{T_2+t}(x) dx dz \\
    &=  \theta \tp \Cov(\hmu_{T_2+ t}) \theta / \sigma^2.
  \end{align*}
  We deduce that $\Var(\mo_0) = 1$ as $\sigma^2 =  \theta \tp \Cov(\hmu_{T_2+t}) \theta$.
  For any $z \in \real$, we have
  \begin{align*}
    d \mo_t(z) &= \int_{\theta \tp x/\sigma  = z} d \hmu_{T_2+t}(x) dx \\
    &= \int_{\theta \tp x/ \sigma  = z} (x- \bb(\hmu_{T_2+t}))\tp \theta \theta\tp dW_t \hmu_{T_2+t}(x) dx \\
    &= \int_{\theta \tp x /\sigma  = z} \sigma (z - \bb(\mo_t)) \theta \tp dW_t \hmu_{T_2+t}(x) dx \\
    &= \sigma (z - \bb(\mo_t)) \theta \tp dW_t \mo_t(z).
  \end{align*}
  Based on the above calculation, $(\mo_t)_{t \geq 0}$ undergoes a 1-dimensional SL and the amount of Gaussian that it makes appear at time $t$ is $t \sigma^2$.

  Based on the $h_S$ definition, we can express $\hmu_{T_2+t}(S)$ using $\mo_t$ and $h_S$ as follows
  \begin{align*}
    \mo_{t}(h_S) &= \int h_S(z) \mo_t(z) dz \\
    &= \int  \frac{\int_{\theta \tp x / \sigma = z}  \mathbf{1}_{x \in S} \hmu_{T_2+t}(x) dx}{\mo_t(z)} \mo_t(z)  dz \\
    &= \hmu_{T_2+t}(S).
  \end{align*}
\end{proof}

\begin{proof}[Proof of Lemma~\ref{lem:lipschitzness_of_G}]
  First, we look at the derivative of $\mo_{y, \thirdT}$ with respect to $y$. Recall that
  \begin{align*}
    \mo_{y, \thirdT} = \frac{\exp(-\frac{\thirdT \sigma^2}{2} \abss{x-y}^2) \nu(x) }{ \int \exp (-\frac{\thirdT \sigma^2}{2} \abss{z-y}^2) \nu(z) dz}.
  \end{align*}
  Taking derivative with respect to $y$, we obtain
  \begin{align*}
    \frac{\partial \mo_{y, \thirdT}(x)}{\partial y} &= - \thirdT\sigma^2 \brackets{ (y-x) - \frac{\int (y- z) \exp (-\frac{\thirdT \sigma^2}{2} \abss{z-y}^2) \mo(z) dz}{\int \exp (-\frac{\thirdT \sigma^2}{2} \abss{z-y}^2) \mo(z) dz} } \mo_{y, \thirdT} \\
    &= \thirdT\sigma^2 [x - \bb(\mo_{y, \thirdT})] \mo_{y, \thirdT}(x).
  \end{align*}
  Consequently, 
  \begin{align*}
    \frac{\partial \mo_{y, \thirdT}(h)}{\partial y} = \int h(x) \cdot \thirdT \sigma^2 [x - \bb(\mo_{y, \alpha})] \mo_{y, \alpha}(x) dx.  
  \end{align*}
  Second, let $v$ be the unit vector in the direction of $y - \tilde{y}$. Then $y - \tilde{y} = \abss{y  - \tilde{y}} v$. Applying the mean value theorem, there exists $\hat{y}$ such that 
  \begin{align*}
    \abss{\mo_{y, \thirdT}(h) - \mo_{\tilde y, \thirdT}(h)} &= \abss{(y - \tilde{y})\tp \frac{\partial \mo_{z, \thirdT}(h)}{\partial z} |_{z=\hat y} } \\
    &= \thirdT\sigma^2 \abss{y  - \tilde{y}} \abss{\int h(x) v \cdot (x - \bb(\mo_{\hat{y}, \thirdT})) \mo_{\hat{y}, \thirdT}(x) dx} \\
    &\leq \thirdT\sigma^2 \abss{y - \tilde{y}} \brackets{ \int \parenth{x - \bb(\mo_{\hat{y}, \thirdT}) }^2 \mo_{\hat{y}, \thirdT}(x) dx   }^{1/2} \brackets{ \int h(x)^2  \mo_{\hat{y}, \thirdT}(x) dx}^{1/2} \\
    &\leq \thirdT\sigma^2 \abss{y - \tilde{y}} \Var(\mo_{\hat{y}, \thirdT})^{1/2} \cdot 1 \\
    &\leq \thirdT\sigma^2 \abss{y - \tilde{y}} \frac{1}{(\thirdT \sigma^2)^{1/2}} \\
    &= \parenth{\thirdT\sigma^2}^{1/2}  \abss{y - \tilde{y}}.
  \end{align*}
  The last inequality follows because $\mo_{y, \thirdT}$ is $\thirdT\sigma^2$-strongly logconcave for any $y$. This completes the proof.
\end{proof}

Before we can prove Lemma~\ref{lem:case2_nu_t_y_Kr_large}, we need the following fact regarding on two logconcave densities. 
\begin{lemma}
  \label{lem:interval_overlap_two_close_logconcave}
  Let $p, \tilde{p}$ be two logconcave densities on $\real$ such that $p$ is isotropic and $\tvdis(p, \tilde{p}) < \epsilon$. Let $0< \delta < 1$. Set $a$ and $b$ to be the $\delta$ and $(1-\delta)$-quantiles respectively of $p$, respectively. Suppose further that $\epsilon \leq \delta^2 / 10^5$. Then for all $z \in [a, b]$, we have $\frac{\tilde{p}}{p}(z) > 0.9$.
\end{lemma}

\begin{proof}[Proof of Lemma~\ref{lem:interval_overlap_two_close_logconcave}]
  Without loss of generality (by convolving with small Gaussian and taking limit), we can assume that $\frac{\tilde{p}}{p}$ is well-defined on $\real$ and both densities are $C^2$-smooth. Let $\tilde{a}$ and $\tilde{b}$ be the $\delta/2$ and $(1-\delta/2)$-quantiles of $p$. Define $f(\cdot) \defn \frac{\tilde{p}}{p}(\cdot)$.

  Suppose there is $z_0 \in [a, b]$ such that $f(z_0) \leq 0.9$. Define $w_- \defn \sup \braces{z \in [\tilde{a}, z_0] \mid f(z) \geq 0.95 }$, noting that such $w_- \in [\tilde a, z_0]$ must indeed exist, for otherwise we always have $f(x) \leq 0.95$ for $x \in [\tilde{a}, a]$, which implies
  \begin{align*}
    \epsilon > \tvdis(\nu, \tilde{\nu}) \geq \int_a^{\tilde a} (1-f(x)) p(x) dx \geq \frac{\delta}{2} \cdot 0.05,
  \end{align*} 
  which contradicts the assumption $\epsilon \leq \delta^2 / 10^5$ and $\delta < 1$. The same reasoning allows us to define
  \begin{align*}
    w_+ \defn \inf \braces{z \in [z_0, \tilde{b}] \mid f(z) \geq 0.95}.
  \end{align*}
  By the definition of $w_-$ and $w_+$, we have $f(x) \leq 0.95$ for all $x \in [w_-, w_+]$, which implies that
  \begin{align*}
    p([w_-, w_+]) \cdot 0.05 \leq \int_{w_-}^{w_+} (1-f(x))p(x) dx \leq \tvdis(p, \tilde{p}) < \epsilon.
  \end{align*}
  Using Lemma~\ref{lem:isotropic_1d_delta_quantile_large} below, we deduce that $p(z) \geq \frac{1}{16e}\delta$ for $z \in [w_-, w_+]$ which implies that
  \begin{align}
    \abss{w_+ - w_- } \leq \frac{320 e \cdot \epsilon}{\delta}.
  \end{align}
  By the mean value theorem, there exists $u_{-} \in [w_-, z_0]$ such that 
  \begin{align*}
    \parenth{\log \tilde{p} - \log p}'(u_-) &= \frac{\log f (z_0) -  \log f (w_-)}{\abss{z_0 - w_-}} \\
    &\leq \frac{ - \log \frac{0.95}{0.9} }{\abss{z_0 - w_-}} \\
    &\leq \frac{ - \log \frac{0.95}{0.9} \cdot \delta }{320 e \cdot \epsilon} =: -\mathfrak{D}.
  \end{align*}
  Similarly, there exist $u_+$ between $z_0$ and $w_+$ such that
  \begin{align*}
    \parenth{\log \tilde{p} - \log p}'(u_+) \geq \mathfrak{D}.
  \end{align*}
  Combining the two bounds above, we have
  \begin{align*}
    \parenth{\log \tilde{p} - \log p}'(u_+) - \parenth{\log \tilde{p} - \log p}'(u_-) \geq 2 \mathfrak{D}.
  \end{align*}
  Since $\log \tilde{p}$ is concave, its derivative is non-decreasing, and consequently we have
  \begin{align*}
    \parenth{\log \tilde{p}}'(u_+) - \parenth{\log \tilde{p} }'(u_-) \leq 0.
  \end{align*}
  It follows that  
  \begin{align*}
     \parenth{\log p }'(u_-) - \parenth{\log p}'(u_+) \geq 2 \mathfrak{D}.
  \end{align*}
  Then either $\parenth{\log p }'(u_-)$ or $- \parenth{\log p}'(u_+)$ has to be larger than $\mathfrak{D}$. If $-\parenth{\log p}'(u_+)  \geq \mathfrak{D}$, because $\log(p)$ is non-increasing, for $z \geq u_+$, 
  \begin{align*}
    \parenth{\log p}'(x) \leq  - \mathfrak{D}.
  \end{align*}
  Integrating leads to the bound $p(z) \leq p(u_+) e^{- \mathfrak{D}(z - u_+)}$. Integrating again, we have 
  \begin{align*}
    \int_{u_+}^ \infty p(z) dz \leq p(u_+) \frac{1}{\mathfrak{D}} \leq \frac{1}{\mathfrak{D}} = \frac{320 e \cdot \epsilon}{\log \frac{0.95}{0.9} \cdot \delta} \leq \frac{20000 \epsilon}{\delta} \leq \frac{1}{5} \delta,
  \end{align*}
  where $p(u_+) \leq 1$ follows from Lemma~\ref{lem:isotropic_1d_max}. This leads to a contradiction to the assumption that the mass on the right of $u_+$ is at least $\frac{1}{2} \delta$. The other case $\parenth{\log p }'(u_-) \geq \mathfrak{D}$ is similar.
\end{proof}

\begin{proof}[Proof of Lemma~\ref{lem:case2_nu_t_y_Kr_large}]
  The proof is split into two disjoint cases: $\frac{1}{512} \leq \thirdT \sigma^2 < \frac{400 \log(10^4 /\xi)}{\xi^2}$ and $\thirdT \sigma^2 \geq \frac{400 \log(10^4 /\xi)}{\xi^2}$. 
  \paragraph{Case 1:} Assume that $\frac{1}{512} \leq \thirdT \sigma^2 < \frac{400 \log(10^4 /\xi)}{\xi^2}$.

    Let $\delta = \xi/20$. Let $a$ and $b$ be the $\delta$ and $(1-\delta)$-quantiles of $\rho$. In this case, we show that the interval $I=[a,b]$ satisfies the requirements of the lemma. In this case we trivially have $\rho(I) \geq 1-\xi/10$.

    Since $\rho$ is the convolution of $\mo$ which has variance $1$ and $\Normal\parenth{0, \frac{1}{\alpha \sigma^2}}$ which has variance $512$, the variance of $\rho$ is upper-bounded by $513$. 
  

   Applying the tail bound for isotropic log-concave density in Lemma~\ref{lem:isotropic_logconcave_tail}, we obtain 
  \begin{align*}
    \abss{a - b} \leq 2 \sqrt{513} \log \frac{e}{\delta} \leq 100 \parenth{4 + \log \frac{1}{\xi}}.
  \end{align*}

  It remains to show that $\omega_{y,\alpha}(h_{K_r}) \geq 7/8$ for all $y \in I$. Let $J \defn \braces{y \in [a, b] \mid \mo_{y, \thirdT}(h_{K_r}) \geq \frac{15}{16} }$. Note that $J$ is not necessarily convex. Since $\Exs_{Y\sim \rho}{\mo_{Y, \thirdT}(K_r)} = \mo_{0}(K_r) \geq 1- \zeta/10$,  we have by Markov's inequality that
  \begin{align}
    \label{eq:case2_p_Kr_large}
    \Prob_{Y \sim \rho} \left (\mo_{Y, \thirdT}(K_r) \geq \frac{15}{16} \right ) \geq 1 - 2 \zeta.
  \end{align}

  Fix any $z_0 \in [a, b] \setminus J$, define the largest interval around $z_0$ contained in $[a, b] \setminus J$ as follows
  \begin{align*}
    z_+ &= \sup \braces{\tilde{z} \mid [z_0, \tilde{z}] \subseteq [a, b] \setminus J } \\
    z_- &= \inf \braces{\tilde{z} \mid [\tilde{z}, z_0] \subseteq [a, b] \setminus J }.
  \end{align*}
  For any $\tilde{z} \in [z_-, z_+]$, we have $\mo_{\tilde{z}, \thirdT}(h_{K_r}) < \frac{15}{16}$ according to the definition of $J$.
  It follows from Eq.~\eqref{eq:case2_p_Kr_large} that $\rho([z_-, z_+]) \leq 2\zeta$. On the other hand, since the variance of $\rho$ is bounded by $513$, applying Lemma~\ref{lem:isotropic_1d_delta_quantile_large}, we obtain $\rho(z) \geq \frac{1}{8 \sqrt{513} e} \delta$ for $z \in [a, b]$. Hence,
  \begin{align*}
    \abss{z_+ - z_-} \leq \frac{2\zeta}{\frac{1}{8 \sqrt{513} e} \delta} \leq \frac{2000\zeta}{\delta}.
  \end{align*}

  Finally, we can use the Lipschitz property of $y \mapsto \mo_{y, \thirdT}(h_{K_r})$ to lower bound $\mo_{\tilde{z}, \thirdT}(h_{K_r})$ for $\tilde{z} \in [z_-, z_+]$. According to Lemma~\ref{lem:lipschitzness_of_G}, $y \mapsto \mo_{y, \thirdT}(h_{K_r})$ is $\sqrt{\alpha \sigma^2} \leq \sqrt{\frac{400 \log(10^4 /\xi)}{\xi^2} }$-Lipschitz. Note that the assumption that $\zeta \leq \xi^2/\sqrt{\log(10^4/\xi)}/10^8$ implies $\frac{2000 \zeta}{ \delta} \cdot \frac{\sqrt{400 \log(10^4/\xi)}}{\xi} \leq \frac{1}{16}$, which in turn ensures that
  \begin{align*}
    \mo_{\tilde{z}, \thirdT}(h_{K_r}) \geq 7/8, \text{ for } \tilde{z} \in [z_-, z_+].
  \end{align*}
  Hence, for any $y \in [a, b]$, we have $\mo_{y, \thirdT}(h_{K_r}) \geq 7/8$. This concludes the case $\frac{1}{512} \leq \thirdT \sigma^2 \leq \frac{400 \log(10^4 /\xi)}{\xi^2}$ with the choice of the interval $I = [a, b]$. 

  \paragraph{Case 2:} Assume that $ \thirdT \sigma^2 \geq \frac{400 \log(10^4 /\xi)}{\xi^2}$. \\
  Take $\delta = \xi/ 100$. Let $a$ and $b$ be the $\delta$ and $(1-\delta)$-quantiles of $\mo$, respectively.
  Let $I \defn [a + \delta, b - \delta]$. Using the same reasoning as in the beginning of the previous case, we know that its length satisfies $\abss{I} \leq \abss{a - b} \leq 100 \parenth{4 + \log \frac1 \xi}$.

  Next, let us show that $\rho(I) \geq 1-\xi / 10$. To that end, let $X \sim \mo$ and $Z \sim \Normal(0, \frac{1}{\thirdT \sigma^2} )$ be independent of each other. Then 
  \begin{align*}
    \rho([a+ \delta, b- \delta]) &= \Prob(X + Z \in [a+ \delta, b- \delta]) \\
    &\geq \Prob(X \in [a + 2\delta, b - 2\delta] \text{ and } Z \in [-\delta, \delta]) \\
    &=  \Prob(X \in [a + 2\delta, b - 2\delta]) \cdot \Prob(Z \in [-\delta, \delta]) \\
    &\overset{(v)}{\geq} (\mo([a, b]) - 4\delta \cdot 1) \cdot \parenth{1 - \delta/500}\\
    &\geq 1- 7\delta \geq 1 - \xi/10 \\
  \end{align*}
  The step (v) follows because $\mo([a, b]) = 1-2\delta$, $\mo[a + 2\delta, b - 2\delta] \geq \mo([a,b]) - 4\delta \cdot 1$ using Lemma~\ref{lem:isotropic_1d_max} and by the Gaussian tail bound in Eq.~\eqref{eq:standard_Gaussian_tail_bound}. 

  To conclude the proof it remains to show that for $y \in I$ one has $\omega_{y,\alpha} (h_{K_r}) \geq 7/8$, where
  \begin{align*}
    \mo_{y, \thirdT}(h_{K_r}) &= \frac{\int \mo(z) \exp(-\frac{\thirdT\sigma^2}{2} \abss{z-y}^2) h_{K_r}(z) dz }{\int \mo(z) \exp(-\frac{\thirdT\sigma^2}{2} \abss{z-y}^2)dz} \\
    & = \hmu_{T_2}(K_r) \frac{\int \tilde \mo(z) \exp(-\frac{\thirdT\sigma^2}{2} \abss{z-y}^2)  dz }{\int \mo(z) \exp(-\frac{\thirdT\sigma^2}{2} \abss{z-y}^2)dz},
  \end{align*}
  and where $\tilde{\mo}(x) = \frac{\mo(x) h_{K_r}(x)}{\hat \mu_{T_2}(K_r)}$. 
  
  Observe that $\tilde{\mo}$ is, by definition, the push-forward of the measure $\tilde{\mu}_{T_2} \defn \hmu_{T_2} \cdot \mathbf{1}_{K_r}/ \hmu_{T_2}(K_r)$ via $x \mapsto \frac{1}{\sigma} \cdot x\tp \theta$. Since by construction $\tvdis(\hmu_{T_2}, \tilde{\mu}_{T_2}) \leq 1- \hmu_{T_2}(K_r) \leq \zeta/10$, we obtain
  \begin{align*}
    \tvdis(\mo, \tilde{\mo}) \leq \tvdis(\hmu_{T_2}, \tilde{\mu}_{T_2}) \leq \zeta/10.
  \end{align*}
  By the Pr\'ekopa-Leindler inequality, we have that $\tilde \omega$ is logconcave, thus we may apply Lemma~\ref{lem:interval_overlap_two_close_logconcave} with $\epsilon = \zeta/10 \leq \delta^2/10^5$, we deduce that 
  \begin{equation} \label{eq:wtildew}
  \frac{\tilde{\mo}}{\mo}(z) \geq 0.9, ~~ \forall z \in [a,b].
  \end{equation}
  
  Recall that the standard Gaussian tail bound for $Z \sim \Normal(0, \frac{1}{\thirdT \sigma^2})$ implies for $\eta > 0$,
  \begin{align}
    \label{eq:standard_Gaussian_tail_bound}
    \Prob \left ( \abss{Z} \geq \frac{\eta}{\sqrt{\thirdT \sigma^2}} \right ) \leq 2 e^{-\eta^2/2}.
  \end{align}
  Take $\eta = 2 \log^{1/2}(\frac{1000}{\delta})$, then $e^{-\eta^2/2} \leq \delta / 1000$. 
  Note that the assumption $\thirdT \sigma^2 \geq \frac{400 \log(10^4 /\xi)}{\xi^2}$ ensures that $\frac{\eta}{\sqrt{\thirdT \sigma^2}} \leq \delta$. We have
  \begin{align*}
    M_1 &\defn \int_{[y - \delta, y+\delta]} \mo(z) \frac{1}{\sqrt{2\pi \frac{1}{\thirdT \sigma^2} }}\exp(-\frac{\thirdT \sigma^2}{2} \abss{z-y}^2) dz \\
    &\overset{(i)}{\geq} \frac{1}{8e}\delta  \int_{[y - \delta, y+\delta]} \mo(z) \frac{1}{\sqrt{2\pi \frac{1}{\thirdT \sigma^2} }}\exp(-\frac{\thirdT \sigma^2}{2} \abss{z-y}^2) dz \\
    & \geq \frac{1}{8e}\delta (1 - 2 e^{-\eta^2/2}) \geq \delta/50 ,
  \end{align*}
  where (i) follows from Lemma~\ref{lem:isotropic_1d_delta_quantile_large}.

  \begin{align}
    M_2 &\defn \int_{[y - \delta, y+\delta]^c} \mo(z) \frac{1}{\sqrt{2\pi \frac{1}{\thirdT \sigma^2} }}\exp(-\frac{\thirdT \sigma^2}{2} \abss{z-y}^2) dz \\
    &\overset{(ii)}{\leq}  \int_{[y - \delta, y+\delta]^c } \frac{1}{\sqrt{2\pi \frac{1}{\thirdT \sigma^2} }}\exp(-\frac{\thirdT \sigma^2}{2} \abss{z-y}^2) dz \\
    &\leq 2 e^{-\eta^2/2} \leq \delta/500,
  \end{align}
  where (ii) follows from Lemma~\ref{lem:isotropic_1d_max}. The two above displays imply that
  \begin{equation} \label{eq:wconc}
  \int \mo(z) \exp\left (-\frac{\thirdT\sigma^2}{2} \abss{z-y}^2 \right ) dz \leq \frac{10}{9} \int_{[y - \delta, y+\delta]} \mo(z) \exp\left (-\frac{\thirdT\sigma^2}{2} \abss{z-y}^2 \right )dz.
  \end{equation}
  Hence, for $y\in [a + \delta, b - \delta]$, we have
    \begin{align*}
    \mo_{y, \thirdT}(h_{K_r}) 
    &= \hmu_{T_2}(K_r) \frac{\int \tilde \mo(z) \exp(-\frac{\thirdT\sigma^2}{2} \abss{z-y}^2)  dz }{\int \mo(z) \exp(-\frac{\thirdT\sigma^2}{2} \abss{z-y}^2)dz} \\
    & \stackrel{ \eqref{eq:wconc} }{\geq} 0.9 \hmu_{T_2}(K_r) \frac{\int_{[y - \delta, y + \delta]} \tilde \mo(z) \exp(-\frac{\thirdT\sigma^2}{2} \abss{z-y}^2)  dz }{\int_{[y - \delta, y + \delta]} \mo(z) \exp(-\frac{\thirdT\sigma^2}{2} \abss{z-y}^2)dz} \\
    & \stackrel{ \eqref{eq:wtildew}}{\geq} 0.9^2 \hmu_{T_2}(K_r) \geq \frac{7}{8}.
  \end{align*}
  This completes the proof.
\end{proof}

\subsection{A uniqueness result for the distribution attained by stochastic localization}
\label{sub:identify_3_stage_SL}
In this subsection, we prove Lemma~\ref{lem:identification_of_marginal_on_Y}, which follows from the inversion of the multidimensional Mellin transform~\cite{antipova2007inversion} (also related to the inverse Laplace transform). To see how the problem is related to the inversion of the multidimensional Mellin transform, we first provide some background. 

The Mellin transform~\cite{mellin1896fundamentale} (see also~\cite{antipova2007inversion}) of a function $\Phi(x)$ defined in the positive orthant $\real^n_+$ is given by the integral
\begin{align*}
  \Mellin[\Phi](z) = \int_{\real^n_+} \Phi(x) x^{z-1} dx. 
\end{align*}

Suppose we know that for $t > 0$,
\begin{align*}
  \mu(x) = \int_{\real^n} \nu_{y, t}(x) p(y) dy, \forall x \in \real^n,
\end{align*}
where $\nu_{y, t}(x) = \frac{\exp\parenth{-\frac{t}{2} \abss{x-y}^2} \mu(x)}{\int \exp\parenth{-\frac{t}{2} \abss{z-y}^2} \mu(z) dz }$. Suppose $\tilde{p}$ also satisfy 
\begin{align*}
  \mu(x) = \int_{\real^n} \nu_{y, t}(x) \tilde{p}(y) dy, \forall x \in \real^n.
\end{align*}
Taking the difference between the two and rearranging the terms that depending on $y$ and $x$ in the integral separately, we obtain 
\begin{align*}
  0 = \mu(x) \exp \parenth{-\frac{t}{2} \abss{x}^2} \int_{\real^n } \frac{ (p(y) - \tilde{p}(y))  \exp \parenth{-\frac{t}{2} \abss{y}^2 } }{\int \exp\parenth{-\frac{t}{2} \abss{z-y}^2} \mu(z) dz  } \cdot \exp (t x \cdot y) dy.
\end{align*}
Using the change of variable $w = \exp(ty) \in \real^n_+$, we obtain
\begin{align*}
  0 = \mu(x) \exp \parenth{-\frac{t}{2} \abss{x}^2} \int_{\real^n_+ } \frac{ \brackets{p\parenth{\frac{1}{t} \log w } - \tilde{p}\parenth{\frac{1}{t} \log w} }\exp \parenth{-\frac{1}{2t} \abss{\log w}^2 } }{\int \exp\parenth{-\frac{t}{2} \abss{z- \frac{1}{t} \log w}^2} \mu(z) dz  } \cdot \frac{1}{t} w^{x-1}  dw.
\end{align*}
Since $\mu(x) \neq 0$ for $x \in K$, we have that for $x \in K$,
\begin{align}
  \label{eq:identify_marginal_transformed}
  0 = \int_{\real^n_+ } \frac{\brackets{p\parenth{\frac{1}{t} \log w } - \tilde{p}\parenth{\frac{1}{t} \log w} } \exp \parenth{-\frac{1}{2t} \abss{\log w}^2 } }{\int \exp\parenth{-\frac{t}{2} \abss{z- \frac{1}{t} \log w}^2} \mu(z) dz  } \cdot w^{x-1}  dw.
\end{align}

It is now clear that $x \mapsto  0$ is the Mellin transform of $w \mapsto \frac{\brackets{p\parenth{\frac{1}{t} \log w } - \tilde{p}\parenth{\frac{1}{t} \log w} } \exp \parenth{-\frac{1}{2t} \abss{\log w}^2 } }{\int \exp\parenth{-\frac{t}{2} \abss{z- \frac{1}{t} \log w}^2} \mu(z) dz  }$. Since $w \mapsto \frac{1}{t} \log w$ is a one-to-one mapping from $\real^n_+$ to $\real^n$, in order to show that $p$ is uniquely defined, it is sufficient to show that the inversion of the multidimensional Mellin transform is well-defined. 

\begin{proof}[Proof of Lemma~\ref{lem:identification_of_marginal_on_Y}] The main strategy is to verify the conditions and then to apply Theorem 2 in~\cite{antipova2007inversion} on the inversion of the Mellin transform for the functions appeared in Eq.~\eqref{eq:identify_marginal_transformed}. In the notation of~\cite{antipova2007inversion}, take the convex set $U \defn K \subset \real^n$ and the convex set $\Theta \defn \ball(0, 2\pi)$, and set
  \begin{align*}
    \mathcal{F}(\cdot) \defn 0
  \end{align*}
  on $ U + i\real^n$. $\mathcal{F}$ is holomorphic on $U + i\real^n$ and bounded. Applying Theorem 2 in~\cite{antipova2007inversion}, its inverse Mellin transform is well-defined, and it is $0$ almost everywhere. We conclude that the function 
  \begin{align*}
    w \mapsto \frac{\brackets{p\parenth{\frac{1}{t} \log w } - \tilde{p}\parenth{\frac{1}{t} \log w} } \exp \parenth{-\frac{1}{2t} \abss{\log w}^2 } }{\int \exp\parenth{-\frac{t}{2} \abss{z- \frac{1}{t} \log w}^2} \mu(z) dz}
  \end{align*}
  is $0$ almost everywhere, which implies $p = \tilde{p}$ almost everywhere.  
\end{proof}

\section{Conductance for the transformed density}
\label{sec:conductance_for_transformed}
In this section we first prove Lemma~\ref{lem:distance_pt_to_center} and then prove Lemma~\ref{lem:conductance_transformed_density}.

\subsection{Distance of a typical point from \texorpdfstring{$\mu_t$}{mu} to its center}
\label{sub:distance_pt_to_center}
The main idea to prove Lemma~\ref{lem:distance_pt_to_center} is to first use the observation in~\cite{el2022information} on the distribution of $c_n$ and then apply the standard Gaussian concentration. 
\begin{proof}[Proof of Lemma~\ref{lem:distance_pt_to_center}]
   Recall that $\mu_n$ is uniquely defined given $c_n$. The first part of the lemma is a direct application of Theorem 2 in~\cite{el2022information}. Additionally, given the data generation $X \sim \mu, Z \sim \Normal(0, \frac{1}{n} \Ind_n)$ independent, and $c_n/n \sim X + Z$, a point drawn from $\mu_t$ has the same law as the conditional distribution
  \begin{align*}
    X \mid c_n. 
  \end{align*}
  We are interested in the conditional distribution $X - \frac{c_n}{n} \mid c_n$. Applying the standard chi-square tail bound (see Lemma 1 in~\cite{laurent2000adaptive}), we obtain the unconditional bound 
  \begin{align*}
    \Prob(\abss{Z}^2 \geq 2 ) \leq e^{-\frac{n}{16} } \notag \\
    \Prob(\abss{Z}^2 \leq \frac{1}{2} ) \leq e^{-\frac{n}{16} }.
  \end{align*}
  Hence,
  \begin{align}
    \label{eq:chi-square_tail_bound}
    \Prob(\abss{Z} \geq \sqrt{2} \text{ or } \abss{Z} \leq \frac{\sqrt{2}}{2})  \leq 2 e^{-\frac{n}{16} }.
  \end{align}
  Let $\mathfrak{E}$ denote the event $\mathfrak{E} \defn \braces{a \in \real^n \mid \Prob( \abss{Z} > \sqrt{2} \text{ or }  \abss{Z} \leq \frac{\sqrt2}{2}\mid c_n = a) > e^{-\frac{n}{32}}}$. Writing out Eq.~\eqref{eq:chi-square_tail_bound} with conditional probability, we obtain
  \begin{align*}
    2 e^{-\frac{n}{16}} \geq \Prob(\abss{Z} \geq \sqrt{2} \text{ or } \abss{Z} \leq \frac{\sqrt{2}}{2}) \geq \Prob( \abss{Z} \geq \sqrt{2} \text{ or } \abss{Z} \leq \frac{\sqrt{2}}{2} \mid c_n \in \mathfrak{E} ) \cdot \Prob(\mathfrak{E}) > e^{-\frac{n}{32}} \Prob(\mathfrak{E}).
  \end{align*}
  Hence, $\Prob(\mathfrak{E}) < 2e^{-\frac{n}{32}}$.
\end{proof}

\subsection{Overlap bound: Proof of Lemma \ref{lem:conductance_transformed_density}}
To lower bound the $s$-conductance in Lemma~\ref{lem:conductance_transformed_density}, we first need to bound the transition overlap for close points in $K_r$ defined in Eq.~\eqref{eq:def_Kr}. This concept of transition overlap was previously proposed in~\cite{lovasz1999hit}. In order to introduce the transition overlap, we first define the notion of $1/8$-quantile hit-and-run step-size.
\begin{definition}[$1/8$-quantile hit-and-run step-size]
  \label{def:median_hitandrun_stepsize}
  Given a target density $\nu$ supported on a convex set $K$. For $x \in K$, define $F_x(\nu)$, the median step-size of the hit-and-run with target density $\nu$, as the step-size such that
  \begin{align}
    \label{eq:median_step_size}
    \Prob_{Y \sim P_{x \to \cdot}(\nu)}\parenth{\abss{Y - x} \leq F_x(\nu)} = \frac{1}{8}.
  \end{align}
\end{definition}
The hit-and-run transition kernel $y \mapsto P_{x\to y}$ is continuous on $K$ for any $x \in K$, so the above quantity is well-defined. 

\begin{lemma}
  \label{lem:transition_overlap}
  Fix $\beta\in \real^n$ and the target density $\nu_{\beta, n}$. Let $u, v \in K_r$ such that $\abss{u-\beta} \in \brackets{\frac{1}{\sqrt{2}}, \sqrt{2}}$. Suppose that
  \begin{align*}
    \abss{u - v} < \frac{2}{\sqrt{n}} \min\braces{F_u(\nu_{\beta, n}), 1},
  \end{align*}
  then there exists a universal constant $C \in (0, 1)$ such that 
  \begin{align*}
    \tvdis\parenth{P_{u\to \cdot } (\nu_{\beta, n}), P_{v \to \cdot}(\nu_{\beta, n}) } < 1 - C \min\braces{\sqrt{n} F_u(\nu_{\beta, n}), 1}.
  \end{align*}
\end{lemma}
To prove the transition overlap bound in Lemma~\ref{lem:transition_overlap}, we need to obtain a rough estimate of $F_u$.
\begin{lemma}
  \label{lem:F_u_lower_bound}
  Let $r > 0$. For any $\beta \in \real^n$ and $u \in K_r$ such that $\abss{u-\beta} \leq \sqrt{2} $, we have 
  \begin{align*}
    F_u(\nu_{\beta, n}) \geq \frac{1}{128} \min\braces{2r, \frac{1}{8\sqrt{n}}}.
  \end{align*}
\end{lemma}
Deferring the proof of Lemma~\ref{lem:F_u_lower_bound} to the end, we are ready to prove Lemma~\ref{lem:transition_overlap}.

\begin{proof}[Proof of Lemma~\ref{lem:transition_overlap}]
  For the sake of brevity, throughout the proof we omit the expression $\nu_{\beta, n}$ in our notation (abbreviating $P = P(\nu_{\beta,n})$ for example), since $\nu_{\beta, n}$ is the only target density considered in this proof. 
  It is sufficient to prove that for any $A \subset K$ nonempty measurable subset, we have
  \begin{align*}
    P_{u \to A} - P_{v \to A} < 1 - C \min\braces{\sqrt{n} F_u, 1}
  \end{align*}
  Since $P_{u \to A} \leq 1$, this is implied by showing that
  \begin{equation} \label{eq:ntsPv}
  P_{v \to A} \geq C \min\braces{\sqrt{n} F_u, 1}.
  \end{equation}
  
  We partition $A$ into four parts as follows
  \begin{itemize}
    \item $A_1$ is the part too close to $u$
    \begin{align*}
      A_1 \defn \braces{x \in A: \abss{x - u } < F_u}.
    \end{align*}
    \item $A_2$ is the part which is not almost orthogonal to $u-v$
    \begin{align*}
      A_2 \defn \braces{x \in A: \abss{(x-u)\tp (u-v)} > \frac{2}{\sqrt{n}} \abss{x - u} \cdot \abss{u-v}}.
    \end{align*}
    \item $A_3$ is the part for which the angle $\angle p_u(x) \beta u$ satisfies $\sin(\angle p_u(x) \beta u) > \frac{2}{\sqrt{n}}$, as illustrated in Figure~\ref{fig:uvXa_loc_explain}
    \begin{align*}
      A_3 \defn \braces{x \in A: \sin\parenth{\angle p_u(x) \beta u} > \frac{2}{\sqrt{n}}},
    \end{align*}
    where $p_u(x)$ is the projection of $\beta$ onto the line through $u$ and $x$. We omit the dependency on $x$ and use $p_u$ when the dependency on $x$ is clear. 
    \item $S \defn A \setminus \parenth{A_1 \cup A_2 \cup A_3}$ is the rest. 
  \end{itemize}
  The proof proceeds in two steps:
  \begin{enumerate}
    \item Show that $P_{u \to S} \geq \frac{1}{4}$.
    \item Show that there exists a constant $C'> 0$ such that $P_{v \to S} \geq C' \cdot P_{u \to S}$. According to the definition of $P_{u \to \cdot}$ in Eq.~\eqref{eq:hit-and-run_transition}, we need to 
    \begin{itemize}
      \item show that $\abss{x-v}$ can be upper bounded via $\abss{x - u}$,
      \item show that $\nu(\ell_{vx})$ can be upper bounded via $\nu(\ell_{ux})$.
    \end{itemize}
  \end{enumerate}
  Note that the two steps establish a lower bound on $P_{v \to A}$, which concludes the proof in light of \eqref{eq:ntsPv}.
  
  \paragraph{Step 1.} According to the definition~\eqref{eq:median_step_size} of $F_u$, we have 
  \begin{align*}
    P_{u \to A_1} \leq P_{u \to \ball_u(F_u)} = \frac{1}{8}.
  \end{align*}
  Given $u,v$, $P_{u \to A_2}$ only depends on the uniform distribution on the unit sphere. We evoke the following well-known result~\cite{tkocz2012upper,ball1997elementary} on the area upper bound of a spherical cap of angle $\phi \in (0, \frac{\pi}{2})$,
  \begin{align}
    \label{eq:sphere_cap_area}
    \frac{\area_n(\phi)}{ 2 \area_n(\pi/2)} \leq e^{-n \cos(\phi)^2 /2},
  \end{align}
  where $\area_n(\phi)$ denotes the area of the cap of angle $\phi$ of a unit $n$-sphere (unit sphere in $\real^n$). Applying Eq.~\eqref{eq:sphere_cap_area} with $\cos(\phi) = \frac{2}{\sqrt{n}}$, we have
  \begin{align*}
    P_{u \to A_2} \leq 0.3. 
  \end{align*}
  Similarly, given $u$, $P_{u \to A_3}$ only depends on the uniform distribution on the unit sphere and an application of Eq.~\eqref{eq:sphere_cap_area} implies
  \begin{align*}
    P_{u \to A_3} \leq 0.3. 
  \end{align*}
   Combining the three displays above, we obtain
  \begin{align*}
    P_{u \to S} \geq 1 - \frac{1}{8} - 0.3 - 0.3 \geq \frac{1}{4}. 
  \end{align*}
  \paragraph{Step 2.} First, for $x \in S$, we have 
  \begin{align}
    \label{eq:xu_lowerbound_by_uv}
    \abss{x - u} \overset{(i)}{\geq} F(u) \overset{(ii)}{\geq} \frac{\sqrt{n}}{2} \abss{u - v}.
  \end{align}
  (i) follows from $x \notin A_1$. (ii) follows from the assumption of the lemma. Second, we show that $\abss{x-v}$ can be upper bounded via $\abss{x-u}$ as follows
  \begin{align}
    \label{eq:xv_relate_xu}
    \abss{x-v}^2 &= \abss{x-u}^2 + \abss{u-v}^2 + 2 (x-u) \tp (u-v) \notag \\
    &\overset{(i)}{\leq} \abss{x-u}^2 + \abss{u-v}^2 + \frac{4}{\sqrt{n}} \abss{x-u}\abss{u-v} \notag \\
    &\overset{(ii)}{\leq}  \abss{x-u}^2 + \frac{4}{n} \abss{x-u}^2 + \frac{8}{n} \abss{x-u}^2 \notag \\
    &\leq \parenth{1+\frac{12}{n}} \abss{x-u}^2 \notag \\
    &\leq \parenth{\parenth{1+\frac{6}{n}} \abss{x-u}}^2.
  \end{align}
  (i) follows from $x \notin A_2$. (ii) follows from Eq.~\eqref{eq:xu_lowerbound_by_uv}. Third, let $q_u$ and $q_v$ be the unit vectors parallel to $u-x$ and $v-x$ respectively, we have
  \begin{align}
    \label{eq:nu_density_ratio}
    \frac{\nu(\ell_{vx})}{\nu(\ell_{ux})} &= \frac{\int_{p_v + t q_v \in K} \nu(p_v + t q_v)dt}{\int_{p_u + t q_u \in K} \nu(p_u + t q_u)dt } \notag \\
    &= \frac{\int_{p_v + t q_v \in K} e^{-\frac{n}{2}\abss{p_v-\beta}^2 - \frac{n}{2}t^2 }dt}{\int_{p_u + t q_u \in K} e^{-\frac{n}{2}\abss{p_u-\beta}^2 - \frac{n}{2}t^2 } dt } \notag \\
    &= \underbrace{\frac{e^{-\frac{n}{2}\abss{p_v-\beta}^2 }}{e^{-\frac{n}{2}\abss{p_u- \beta}^2 }}}_{Q_1(x)} \underbrace{\frac{\int_{p_v + t q_v \in K} e^{ - \frac{n}{2}t^2 }dt}{ \int_{p_u + t q_u \in K} e^{- \frac{n}{2}t^2 } dt }}_{Q_2(x)}.
  \end{align}
To derive an upper bound for the ratio $Q_2(x)$ we, roughly speaking, use the fact that the chord $\ell_{ux} \cap K$ contains an interval, centered at $p_u$, whose length is larger than $F_u$. From the definition of $A_3$, we have $\abss{p_u - u} = \sin(\angle p_u(x) \beta u) \abss{u-\beta}  \leq \frac{2}{\sqrt{n}} \abss{u - \beta} \leq \frac{2 \sqrt{2}}{\sqrt{n}}$. From the definition of $A_1$, we have $\abss{x-u} \geq F_u$. If $F_u \leq \frac{2}{\sqrt{n}}$, then $\frac{1}{\sqrt{2\pi/n}} \int_{p_u + t q_u \in K} e^{- \frac{n}{2}t^2 } dt \geq \sqrt{n} F_u \cdot \gamma(2\sqrt{2} + 2)$; otherwise, $\frac{1}{\sqrt{2\pi/n}} \int_{p_u + t q_u \in K} e^{- \frac{n}{2}t^2 } dt \geq \Gamma(2\sqrt{2} + 2) - \Gamma(2\sqrt{2})$.

Combining both cases, we conclude that $R_2$ obeys the bound
  \begin{align}
    \label{eq:R_2_bound}
    Q_2(x) = \frac{\frac{1}{\sqrt{2\pi/n}}\int_{p_v + t q_v \in K} e^{ - \frac{n}{2}t^2 }dt}{\frac{1}{\sqrt{2\pi/n}} \int_{p_u + t q_u \in K} e^{- \frac{n}{2}t^2 } dt } \leq \frac{1}{\min\braces{\sqrt{n} F_u \cdot \gamma(2\sqrt{2} + 2),  \Gamma(2\sqrt{2} + 2) - \Gamma(2\sqrt{2})}}.
  \end{align}
  To upper bound the first ratio $Q_1(x)$, it is sufficient to upper bound $ - \abss{p_v -\beta}^2 + \abss{p_u - \beta}^2$. Let $\delta_1\defn \angle{u\beta v}, \delta_2 \defn \angle{uxv}$. According to the definition of $A_1$, we have
  \begin{align}
    \label{eq:sin_delta2_bound}
    \sin(\delta_2) \leq \frac{\abss{u-v}}{\abss{u-x}} \leq \frac{2}{\sqrt{n}}.
  \end{align}
  From the assumption, we have \begin{equation} \label{eq:uminusv}
      \abss{u-v} \leq \frac{2}{\sqrt{n}},
  \end{equation} and $\abss{u-\beta} \geq \frac{\sqrt{2}}{2}$. Hence,
  \begin{align}
    \label{eq:sin_delta1_bound}
    \sin(\delta_1) \leq \frac{\abss{u-v}}{\abss{u-\beta}} \leq \frac{2\sqrt{2}}{n}.
  \end{align}
  \begin{figure}[ht]
    \centering
    \label{fig:uvXa_loc_explain}
    \includegraphics[width=0.9\textwidth]{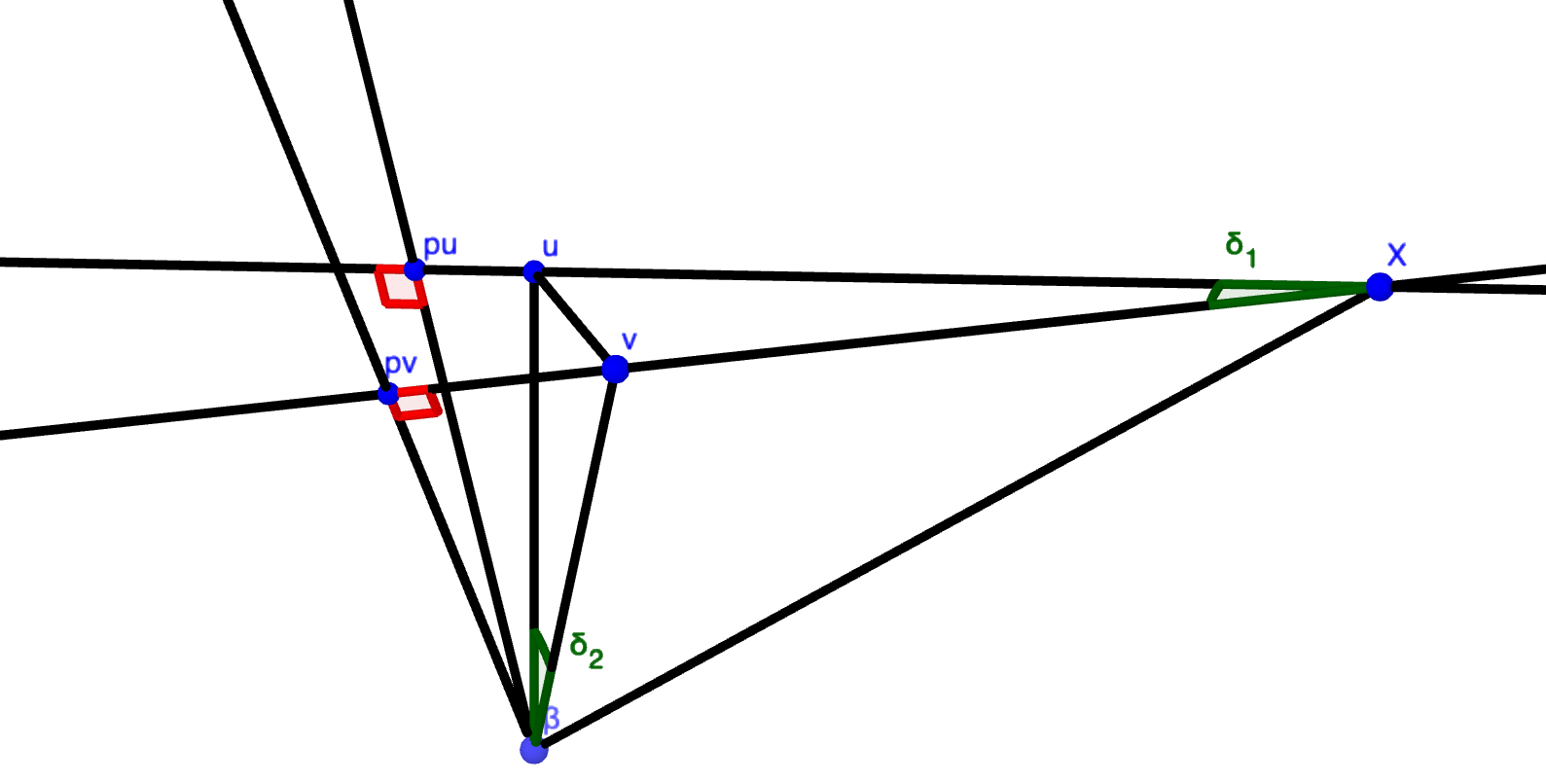}
    \caption{$u, v, \beta, x$ placed in a 3D plot}
  \end{figure}
  To lower-bound $\abss{p_v - \beta}$, we need to upper-bound $\abss{p_v - v}$ and hence the angle $\angle{v \beta p_v}$. Note that $\frac{\pi}{2} - \angle{v \beta p_v} = \angle{p_v v \beta} = \angle{v\beta x} + \angle{vx\beta}$. The two angles $\angle{v\beta x}$ and $\angle{vx\beta}$ can be bounded via $\angle{u\beta x}$ and $\angle{ux\beta}$ respectively. 
  
  Let $\Delta_1 = \angle{u\beta x} - \angle{v \beta x}$ and $\Delta_2 = \angle{ux \beta} - \angle{vx\beta}$. We claim that $|\Delta_i| \leq \delta_i$ for $i=1,2$. Indeed, if $w_1 = \frac{u-\beta}{|u-\beta|}$, $w_2 = \frac{v-\beta}{|v-\beta|}$ and $w_3 = \frac{x-\beta}{|x-\beta|}$ then by the spherical triangle inequality we have 
  $$
  |\arccos(\langle w_1, w_3 \rangle) - \arccos(\langle w_2, w_3 \rangle)| \leq \arccos(\langle w_1, w_2 \rangle),
  $$
  which amounts to the fact that $|\Delta_1| \leq \delta_1$, and an analogous argument shows that $|\Delta_2| \leq \delta_2$.
  We therefore have
  \begin{align*}
    \abss{p_v - v} &= \abss{v-\beta } \sin(\angle{p_v \beta v}) \\
    &= \abss{v-\beta} \sin(\frac{\pi}{2} - \angle{v\beta x} - \angle{vx\beta}) \\
    &= \abss{v-\beta} \sin(\frac{\pi}{2}-\angle{u\beta x} - \angle{ux\beta} + \Delta_1 + \Delta_2) \\
    &= \abss{v-\beta} \sin(\angle{p_u\beta u} + \Delta_1 + \Delta_2) \\
    &= \abss{v-\beta} \parenth{\sin(\angle{p_u \beta u})  + \sin(|\Delta_1| + \sin(|\Delta_2|)  } \\
    &\leq \abss{v-\beta} \parenth{\frac{\abss{p_u - u}}{\abss{u-\beta}} + \abss{\sin(\delta_1)} + \abss{\sin(\delta_2)} }.
  \end{align*}
  Plugging the bounds~\eqref{eq:sin_delta2_bound}~\eqref{eq:sin_delta1_bound} into the above equation, together with the definition of $A_3$, we obtain
  \begin{align*}
    \frac{\abss{p_v - v}}{\abss{v-\beta}} \leq \frac{\abss{p_u - u}}{\abss{u-\beta}} + \frac{2\sqrt{2} + 2}{\sqrt{n}} \leq \frac{10}{\sqrt{n}}.
  \end{align*}
  Together with the assumption $\abss{v-\beta} \leq \sqrt{2}$, we deduce that \begin{align} \label{eq:pvminusv}
      \abss{p_v - v} \leq \frac{10\sqrt{2}}{\sqrt{n}}.
  \end{align}
  Using the fact that $p_u$ and $p_v$ are orthogonal projections, we have
  \begin{align*}
    - \abss{p_v - \beta}^2 + \abss{p_u - \beta}^2 &= - \abss{v-\beta}^2 + \abss{p_v - v}^2 + \abss{u-\beta}^2 - \abss{p_u - u}^2 \\
    & \leq \abss{u-v} \parenth{2\abss{u-\beta} + \abss{u-v}} + \abss{p_v - v}^2 - \abss{p_u - u}^2 \\
    &\leq \abss{u-v} \parenth{2\abss{u-\beta} + \abss{u-v}} + \abss{p_v - v}^2  \\
    & \stackrel{ \eqref{eq:uminusv} \wedge \eqref{eq:pvminusv}}{\leq} \frac{4}{n} \parenth{2 \sqrt{2}  + \frac{4}{n}} + \frac{200}{n} \\
    &\leq \frac{230}{n},
  \end{align*}
  where the first inequality is obtained by the triangle inequality. We therefore have
  \begin{align}
    \label{eq:R_1_bound}
    Q_1(x) = \frac{e^{-\frac{n}{2}\abss{p_v-\beta}^2 }}{e^{-\frac{n}{2}\abss{p_u-\beta}^2 }} \leq e^{115}.
  \end{align}
  Combining the bounds for the two ratios, we obtain
  \begin{align*}
    P_{v \to S} &= \frac{2}{n \pi_n} \int_S \frac{f(x)}{\nu(\ell_{vx}) \abss{x-v}^{n-1}} dx \\
    &\overset{(i)}{\geq} \frac{2}{n \pi_n} \int_S \frac{f(x)}{Q_1(x) Q_2(x) \nu(\ell_{ux}) \abss{x-v}^{n-1}} dx \\
    &\overset{(ii)}{\geq} \frac{2}{e^6 \cdot n \pi_n} \int_S \frac{f(x)}{ Q_1(x) Q_2(x) \nu(\ell_{ux}) \abss{x-u}^{n-1}} dx \\
    & \stackrel{ \eqref{eq:R_2_bound} \wedge \eqref{eq:R_1_bound} }{ \geq } C \min\braces{\sqrt{n} F_u, 1} \cdot P_{u \to S},
  \end{align*}
  where (i) applies Eq.~\eqref{eq:nu_density_ratio}. (ii) follows from~\eqref{eq:xv_relate_xu} and $(1+\frac{6}{n})^n \leq e^6$. Here $C$ is a universal constant that depends on the universal constant that appear in Eq.~\eqref{eq:R_2_bound} and~\eqref{eq:R_1_bound}.

  Finally, we have
  \begin{align*}
    P_{u \to A} - P_{v \to A} &\leq 1 - P_{v \to S} \\
    & \leq 1 - C \min\braces{\sqrt{n} F_u, 1} \cdot P_{u \to S} \\
    & \leq 1 - \frac{1}{4} C \min\braces{\sqrt{n} F_u, 1}.
  \end{align*}
  We conclude. Remark that the constants were not optimized for the sake of the simplicity of the derivation. 
\end{proof}

\begin{proof}[Proof of Lemma~\ref{lem:F_u_lower_bound}]
  The lower bound proof proceeds similarly as that of Lemma 3.2 in~\cite{lovasz2006hit}, except that we have to deal with the Gaussian supported on the convex set $\nu_{\beta, n}$. Define $s: K \to \real^+$ as
  \begin{align}
    \label{eq:def_su}
    s(u) \defn \sup\braces{t \in \real_+  \middle\vert \lambda(u, t) \geq \frac{63}{64}},
  \end{align}
  where $\lambda(u, t)$ is defined in Eq.~\eqref{eq:def_fracInt}. By the above definition of $s(u)$ and since $u \in K_r$, we have $s(u) \geq 2r$. To simplify notation when the dependency on $u$ is clear, we simply write $s \defn s(u)$. Let $\eta$ denote the fraction of the surface of the ball $\ball(u, s/2)$ that is not in $K$. Then using the fact that $K$ is convex, we have
  \begin{align*}
    \vol\parenth{\ball\parenth{u,s} \setminus K } \geq \eta \cdot \vol( \ball\parenth{0, s}) - \vol\parenth{ \ball\parenth{0, s/2}}.
  \end{align*}
  On the other hand, the definition of $s$ implies that
  \begin{align*}
    \vol\parenth{\ball\parenth{u,s} \setminus K } \leq \frac{1}{64} \vol( \ball\parenth{0, s}).
  \end{align*}
  We deduce that for $n\geq 9$, 
  \begin{align*}
    \eta \leq \frac{1}{64} + 2^{-n} \leq \frac{3}{128}.
  \end{align*}
  Take a line $\ell$ through $u$ with direction uniformly distributed on the unit sphere. Then with probability at least $1-2\eta$, $\ell \cap \ball(u, s/2) \subseteq K$. 

  Now, let $p_u$ be the orthogonal projection of the point $\beta$ on the line $\ell$. Then $|p_u-u| = |u-\beta| \cos(\alpha)$ where $\alpha$ is the angle between $\ell$ and the line connecting $u$ and $\beta$. An application of the spherical cap area upper bound as in Eq.~\eqref{eq:sphere_cap_area} with $\cos(\alpha) = \frac{2\sqrt{2}}{\sqrt{n}}$ and a union bound implies that with probability at least $1- 2 \eta - \frac{1}{16}$, we have $\ell \cap \ball(u, s/2) \subseteq K$ and $\abss{p_u - u} \leq  \sqrt{2} \cdot \frac{2\sqrt{2}} {\sqrt{n}}  \leq \frac{4}{\sqrt{n}}$. Define $\tau \defn \min\braces{s, \frac{1}{8\sqrt{n}}}$.
  Then
  \begin{align*}
    \Prob_{Y \sim P_{u \to \cdot}} \parenth{ \abss{Y - u} \leq \frac{\tau}{128} \mid Y \in \ell} &\overset{(i)}{\leq} \frac{\Gamma\parenth{\sqrt{n} \parenth{b + \frac{\tau}{256}}} - \Gamma\parenth{\sqrt{n}\parenth{b - \frac{\tau}{256}}}}{ \Gamma\parenth{\sqrt{n}\parenth{b + \frac{s}{2}}} - \Gamma\parenth{\sqrt{n}\parenth{b - \frac{s}{2}}}} \\
    &\leq \frac{\Gamma\parenth{\sqrt{n} \parenth{b + \frac{\tau}{256}}} - \Gamma\parenth{\sqrt{n}\parenth{b - \frac{\tau}{256}}}}{ \Gamma\parenth{\sqrt{n}\parenth{b + \frac{\tau}{2}}} - \Gamma\parenth{\sqrt{n}\parenth{b - \frac{\tau}{2}}}} \\
    &\overset{\phantom{(ii)}}{\leq} \frac{\frac{\sqrt{n} \tau }{128}  \gamma(\sqrt{n} \parenth{b - \frac{\tau}{256}})}{ \sqrt{n} \tau  \gamma(\sqrt{n}(b + \frac{\tau}{2}))} \\
    &\overset{(ii)}{\leq} \frac{1}{64},
  \end{align*}
  where $b = \abss{p_u - u}$, $\Gamma$ is the cumulative density function  of the standard Gaussian and $\gamma$ is the density function of the standard Gaussian. (i) follows from looking at the one dimensional truncated Gaussian. (ii) follows because $\sqrt{n}b \leq 4$, $\tau \leq \frac{1}{8\sqrt{n}}$, and a numerical calculation shows the ratio $\frac{\gamma(\sqrt{n} \parenth{b - \frac{\tau}{256}})}{ \gamma(\sqrt{n}(b + \frac{\tau}{2}))} \leq 2$.
  For the unconditional probability, we have 
  \begin{align*}
    \Prob_{Y \sim P_{u \to \cdot}}\parenth{ \abss{Y - u} \leq \frac{\tau}{128}} \leq (2\eta + \frac{1}{16}) \cdot 1  + (1-2\eta - \frac{1}{16}) \cdot \frac{1}{64} < \frac{1}{8}.
  \end{align*}
  Hence,
  \begin{align*}
    F_u(\nu_{\beta, n}) \geq \frac{\tau}{128} \geq \frac{1}{128} \min\braces{2r, \frac{1}{8\sqrt{n}}}.
  \end{align*}

\end{proof}

\subsubsection{Proof of Lemma~\ref{lem:conductance_transformed_density}}
\label{sec:proof_of_lem:conductance_transformed_density}
\begin{proof}[Proof of Lemma~\ref{lem:conductance_transformed_density}]
  To simplify notation, let $\nu = \nu_{\beta, n}$. Consider the truncated density $\nu_{\dagger}$ 
  \begin{align*}
    \nu_{\dagger}(x) \defn \mathbf{1}_{K_r}(x) e^{-\frac{n}{2} \abss{x-\beta}^2 }   \frac{1}{\nu(K_r)}
  \end{align*}
  Note that since $K_r$ is convex, $\nu_\dagger$ is still $n$-strongly-logconcave. Consequently, it satisfies the isoperimetric inequality (see Theorem 5.4 in~\cite{cousins2018gaussian}), for $U_1, U_2, U_3$ a partition of $K$, 
  \begin{align*}
    \nu_\dagger(U_3) \geq \log2 \cdot \sqrt{n} \cdot d(U_1, U_2) \cdot \nu_\dagger(U_1) \nu_\dagger(U_2).
  \end{align*}
  \tikzset{every picture/.style={line width=0.75pt}} 
\begin{figure}[ht]
  \label{fig:conductance_lower_bound_partition}
\begin{tikzpicture}[x=0.75pt,y=0.75pt,yscale=-1,xscale=1]

\draw  [color={rgb, 255:red, 65; green, 117; blue, 5 }  ,draw opacity=1 ] (369,27) -- (503,229) -- (64,253) -- (83,128) -- (141,66) -- cycle ;
\draw  [color={rgb, 255:red, 208; green, 2; blue, 27 }  ,draw opacity=1 ] (147,90) .. controls (171,69) and (324,41) .. (365,64) .. controls (406,87) and (470,188) .. (441,214) .. controls (412,240) and (108,241) .. (88,211) .. controls (68,181) and (123,111) .. (147,90) -- cycle ;
\draw    (209,15) .. controls (250,43) and (187,168) .. (245,183) .. controls (303,198) and (298,247) .. (328,283) ;
\draw  [dash pattern={on 4.5pt off 4.5pt}]  (285,54) .. controls (305,94) and (263,153) .. (321,168) .. controls (379,183) and (348,193) .. (378,229) ;
\draw  [dash pattern={on 4.5pt off 4.5pt}]  (162,85) .. controls (182,125) and (128,188) .. (186,203) .. controls (244,218) and (221,200) .. (251,236) ;
\draw  [fill={rgb, 255:red, 245; green, 166; blue, 35 }  ,fill opacity=0.4 ] (383,215.5) -- (409,234.5) -- (64,253) -- (64,253) -- (117,211.5) -- cycle ;

\draw (178,15.4) node [anchor=north west][inner sep=0.75pt]    {$S_{1}$};
\draw (231,15.4) node [anchor=north west][inner sep=0.75pt]    {$S_{2}$};
\draw (346,116.4) node [anchor=north west][inner sep=0.75pt]    {$S'_{2}$};
\draw (118,160.4) node [anchor=north west][inner sep=0.75pt]    {$S'_{1}$};
\draw (479,211.4) node [anchor=north west][inner sep=0.75pt]  [color={rgb, 255:red, 65; green, 117; blue, 5 }  ,opacity=1 ]  {$K$};
\draw (427,191.4) node [anchor=north west][inner sep=0.75pt]  [color={rgb, 255:red, 208; green, 2; blue, 27 }  ,opacity=1 ]  {$K_{r}$};
\draw (101,227.4) node [anchor=north west][inner sep=0.75pt]  [color={rgb, 255:red, 255; green, 255; blue, 255 }  ,opacity=1 ]  {$\Upsilon ^{c}$};

\end{tikzpicture}
\centering
\caption{Illustration of the partition of $K$ in conductance lower bound}
\end{figure}
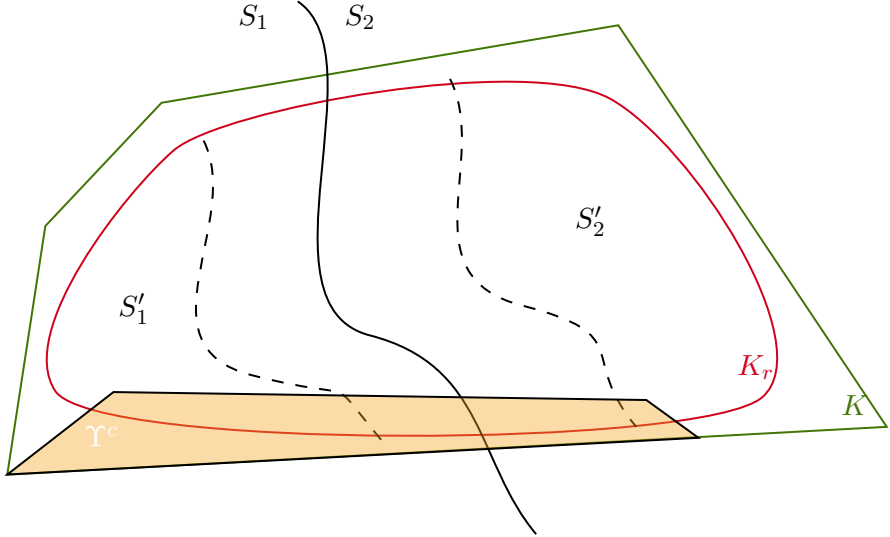
  
  Let $\varrho = \frac{r \sqrt{n}}{64 C}$ where $C$ is the constant from Lemma~\ref{lem:transition_overlap}.
  Define the sets
  \begin{align*}
    S_1' \defn \braces{u \in S_1 \cap K_r \mid P_{u \to S_2}(\nu) < \frac{\varrho}{2}}, \quad S_2' \defn \braces{u \in S_2 \cap K_r  \mid P_{u \to S_1}(\nu) < \frac{\varrho}{2}}.
  \end{align*}
  There are two cases:
  \begin{itemize}
    \item Case 1: $\nu(S_1') \leq \nu(S_1 \cap K_r)/2$ or $\nu(S_2') \leq \nu(S_2 \cap K_r)/2$.
    \item Case 2: $\nu(S_i') \geq \nu(S_i \cap K_r)/2$ for $i=1,2$.
  \end{itemize}
  \paragraph{Case 1:} Assuming that $\nu(S_1 \cap K_r \setminus S_1') \geq \nu(S_1 \cap K_r)/2 $, we obtain
  \begin{align*}
    \int_{S_1} P_{x \to S_2}(\nu) d\nu(x) &\geq \int_{S_1 \cap K_r \setminus S_1'} P_{x \to S_2}(\nu) d\nu(x) \\
    &\overset{(i)}{\geq} \frac{\varrho}{2} \nu(S_1 \cap K_r \setminus S_1') \\
    &\geq \frac{\varrho}{4} \nu(S_1 \cap K_r).
  \end{align*}
  (i) follows from the definition of $S_1'$. The proof is similar for the case where the roles of $S_1$ and $S_2$ are switched, and using the reversibility of the kernel $P$ which implies that $\int_{S_1} P_{x \to S_2}(\nu) d \nu(x) = \int_{S_2} P_{x \to S_1}(\nu) d \nu(x)$.
  \paragraph{Case 2:} For any $u \in S_1' \cap \Upsilon$ and $v \in S_2' \cap \Upsilon$, we have
  \begin{align*}
    \tvdis\parenth{P_{u \to \cdot}(\nu), P_{v \to \cdot}(\nu)} \geq P_{u \to S_1}(\nu) - P_{v \to S_1}(\nu) = 1 - P_{u \to S_2}(\nu) - P_{v \to S_1}(\nu) > 1 - \varrho. 
  \end{align*}
  Lemma~\ref{lem:F_u_lower_bound} implies that for $u \in S_1' \cap \Upsilon$ and $r \leq \frac{1}{16 \sqrt{n}}$, we have
  \begin{align}
    \label{eq:F_u_lower_bound_in_proof}
    F_u(\nu_{\beta, n}) \geq \frac{r}{64}.
  \end{align}
  Together with Lemma~\ref{lem:transition_overlap}, we have that $\abss{u-v} \geq \Delta \defn \frac{r}{32 \sqrt{n}}$. Since it holds for any pair of $u \in S_1' \cap \Upsilon$ and $v \in S_2' \cap \Upsilon$, its implies that 
  \begin{align*}
    d\parenth{S_1' \cap \Upsilon, S_2' \cap \Upsilon} \geq \Delta.
  \end{align*}
  We have
  \begin{align*}
    \int_{S_1} P_{x \to S_2}(\nu) d\nu(x) &= \frac{1}{2} \parenth{\int_{S_1} P_{x \to S_2} d\nu(x) + \int_{S_2} P_{x \to S_1} d\nu(x)} \\
    &\geq \frac{1}{2} \parenth{\int_{S_1 \cap K_r \setminus S_1'} P_{x \to S_2} d\nu(x) + \int_{S_2 \cap K_r \setminus S_2'} P_{x \to S_1} d\nu(x)}  \\
    &\overset{(i)}{\geq} \frac{\varrho}{4} \brackets{\nu(S_1 \cap K_r\setminus S_1') + \nu(S_2 \cap K_r\setminus S_2')} \\
    & = \frac{\varrho}{4} \nu(K_r \setminus (S_1' \cup S_2')).
  \end{align*}
  (i) follows from the definition of $S_1'$ and $S_2'$. Note that the three sets $S_1' \cap \Upsilon, S_2' \cap \Upsilon$ and $K_r \setminus ((S_1' \cup S_2') \cap \Upsilon)$ form a partition of $K_r$. We have
  \begin{align*}
    \nu(K_r \setminus (S_1' \cup S_2')) + \delta  &\geq  \nu(K_r \setminus ((S_1' \cup S_2') \cap \Upsilon)) \\
    &\overset{(i)}{\geq} \frac{\log(2) \cdot \sqrt{n} }{\nu(K_r)}  \cdot d(S_1' \cap \Upsilon, S_2' \cap \Upsilon) \cdot \nu(S_1' \cap \Upsilon) \nu(S_2' \cap \Upsilon) \\
    &\geq \frac{1}{2} \Delta \sqrt{n} \cdot \parenth{ \nu(S_1' \cap \Upsilon)} \nu(S_2' \cap \Upsilon) \\
    &\geq \frac{1}{2} \Delta \sqrt{n} \cdot \parenth{ \nu(S_1')-\delta} \nu(S_2' \cap \Upsilon) \\
    &\geq \frac{1}{2} \Delta \sqrt{n} \cdot \nu(S_1') \nu(S_2' \cap \Upsilon) - \frac12 \Delta \sqrt{n} \delta  \\
    &\geq \frac{1}{2} \Delta \sqrt{n} \cdot \nu(S_1') \nu(S_2') - \Delta \sqrt{n} \delta  \\
    &\overset{(ii)}{\geq} \frac{1}{8} \Delta \sqrt{n} \cdot \nu(S_1 \cap K_r) \cdot \nu(S_2 \cap K_r)  - \Delta \sqrt{n} \delta.
  \end{align*}
  (i) applies the isoperimetric inequality for $\nu_\dagger$. (ii) applies the condition of Case 2. Combine the above two displays, we conclude there exists a universal constant $C' > 0$ such that
  \begin{align*}
    \int_{S_1} P_{u \to S_2}(\nu) d\nu(x) \geq \frac{r^2 \sqrt{n} } { C' } \brackets{\nu(S_1 \cap K_r) \cdot \nu(S_2 \cap K_r)  - 8 \parenth{1 + \frac{32}{r} }\delta }.
  \end{align*}
\end{proof}

\bibliographystyle{alpha}
\bibliography{ref}

\appendix
\newpage
\section{Summary of properties of a logconcave distribution}
\label{app:properties_of_logconcave}
Here is a list of well-known properties of logconcave distributions.  
\begin{itemize}
  \item The log-concavity of a measure is preserved by affine transformations and by marginalization, see Proposition 3.1 and Theorem 3.3 in~\cite{saumard2014log}.
  \item The strong log-concavity of a measure is preserved by affine transformations, by convolution (Theorem 3.7 in~\cite{saumard2014log}) and by marginalization (Theorem 3.8 in~\cite{saumard2014log}).
  \item The isoperimetric constant of a 1-dimensional isotropic logconcave density is lower bounded by $\log(2)/2 \approx 0.34$ (see Theorem 4.3 in~\cite{vempala2005geometric}). 
  \item The maximal value of a 1-dimensional isotropic logconcave density $p$ is bounded by $1$ (Lemma 5.5 (a) in~\cite{lovasz2007geometry}). It is restated in Lemma~\ref{lem:isotropic_1d_max}.
  \item For an isotropic logconcave density $p$, $p(0) \geq 1/8$ (Lemma 5.5 (b) in~\cite{lovasz2007geometry}). It is restated in Lemma~\ref{lem:isotropic_1d_zero}.
  \item An isotropic logconcave density has an exponential tail (Lemma 5.17 in~\cite{lovasz2007geometry}). It is restated in Lemma~\ref{lem:isotropic_logconcave_tail}.
  \item Let $a, b$ be the $\delta$ and $(1-\delta)$-quantile of $p$, then $p(a) \geq \frac{1}{8e} \delta$ and $p(b) \geq \frac{1}{8e} \delta$  as shown in Lemma~\ref{lem:isotropic_1d_delta_quantile_large}.
  \item $\abss{p'(a)} < \frac{2}{\delta}$ as shown in Lemma~\ref{lem:isotropic_1d_delta_quantile_derivative_bounded}.
\end{itemize}

\begin{lemma}[Lemma 5.5 (a) in~\cite{lovasz2007geometry}]
  \label{lem:isotropic_1d_max}
  Let $p$ be an isotropic logconcave density on $\real$. Then for any $x \in \real$, 
  \begin{align*}
    p(x) \leq 1.
  \end{align*}
\end{lemma}

\begin{lemma}[Lemma 5.5 (b) in~\cite{lovasz2007geometry}]
  \label{lem:isotropic_1d_zero}
  Let $p$ be an isotropic logconcave density on $\real$, then $p(0) \geq \frac{1}{8}$.
\end{lemma}
See Lemma 5.5 in~\cite{lovasz2007geometry} for a proof of the above two lemmas. 

\begin{lemma}
  \label{lem:isotropic_logconcave_tail}
  For any isotropic logconcave density $p$ in $\real^n$, and any $t > 0$, we have 
  \begin{align*}
    \Prob_{X \sim p}(\abss{X} \geq t \sqrt{n}) \leq e^{-t+1}.
  \end{align*}
\end{lemma}
See Lemma 5.17 in~\cite{lovasz2007geometry} for a proof.

\begin{lemma}
  \label{lem:isotropic_1d_delta_quantile_large}
  Let $p$ be an isotropic logconcave density on $\real$ and let $z$ be the $(1-\delta)$-quantile (that is, $\int_{-\infty}^z p(x) dz  = 1-\delta$), with $0 < \delta \leq 1/e$, then $p(z) \geq \frac{1}{8e} \delta$. Similarly, let $w$ be the $\delta$-quantile, then $p(w) \geq \frac{1}{8e} \delta$. Additionally, for any $\tilde{z} \in [w, z]$, we have $p(\tilde{z}) \geq \frac{1}{8e} \delta$.
\end{lemma}
\begin{proof}[Proof of Lemma~\ref{lem:isotropic_1d_delta_quantile_large}]
  According to Lemma 5.4 in~\cite{lovasz2007geometry}, we have
  \begin{align*}
    z \geq 0.
  \end{align*}
  Based on the tail of the logconcave density in Lemma 5.7 in~\cite{lovasz2007geometry}, if $z \geq \log(e/\delta)$, then $z > 1$ and 
  \begin{align*}
    \delta < e^{-z + 1} \leq \delta,
  \end{align*}
  which leads to a contradiction. Hence, $z < \log(e/\delta)$.

  Suppose $p(z) < \frac{p(0)}{e} \delta < p(0)$.
  Applying the mean value theorem, there exists $y \in [0, z]$, such that 
  \begin{align*}
    (\log p)'(y) = \frac{\log p (z) - \log p(0)}{z} \leq -1.
  \end{align*}
  The derivative is non-increasing, as $p$ is logconcave. Thus, for $x \geq z$, we have
  \begin{align*}
    (\log p)'(x) \leq -1.
  \end{align*}
  Integrating from $z$ to $x$, we obtain
  \begin{align*}
    p(x) \leq p(z) e^{- (x - z)}.
  \end{align*}
  Integrating again from $z$ to $\infty$, we obtain
  \begin{align*}
    \delta = \int_{z}^\infty p(x) dx \leq p(z) < \frac{p(0)}{e} \delta \leq \frac{1}{e} \delta,
  \end{align*}
  which is a contradiction according to Lemma~\ref{lem:isotropic_1d_max}. 
  Hence,
  \begin{align*}
    p(z) \geq \frac{p(0)}{e} \delta \geq \frac{1}{8e} \delta,
  \end{align*}
  where the last step follows from Lemma~\ref{lem:isotropic_1d_zero}.

  The proof of the $\delta$-quantile is similar. Additionally, since $(\log p)'(z) \leq -1$, $(\log p)'(w) \geq 1$ and the derivative is non-increasing, we conclude that for any $\tilde{z} \in [w, z]$, we have $p(\tilde{z}) \geq \frac{1}{8e} \delta$.

\end{proof}

\begin{lemma}
  \label{lem:isotropic_1d_delta_quantile_derivative_bounded}
  Let $p$ be an isotropic logconcave density on $\real$ and let $z$ be the $(1-\delta)$-quantile. That is, $\int_{-\infty}^z p(x) dz  = 1-\delta$. Suppose $0 < \delta \leq 1/e$, then there exists a universal constant $c > 0$ such that $p'(z) \geq - \frac{2}{\delta}$. By symmetry, a similar result holds for the $\delta$-quantile. 
\end{lemma}
\begin{proof}[Proof of Lemma~\ref{lem:isotropic_1d_delta_quantile_derivative_bounded}]
Suppose $p'(z) < -\frac{2}{\delta}$. Then because $0 \leq p(z) \leq 1$, 
\begin{align*}
  (\log p)' (z) = \frac{p'(z)}{p(z)} < - \frac{2}{\delta}. 
\end{align*}
The derivative is non-increasing as $p$ is logconcave. Thus, for $x \geq z$, we have 
\begin{align*}
  (\log p)'(x) < -\frac{2}{\delta}.
\end{align*}
Integrating twice, we obtain
\begin{align*}
  \delta = \int_z^{\infty} p(x) dx \leq \int_z^{\infty} e^{-\frac{2}{\delta}(x-z)}dx \leq \frac{\delta}{2},
\end{align*}
which is a contradiction. 
  
\end{proof}

\end{document}